\documentclass[graybox]{svmult}

\usepackage{mathptmx}       
\usepackage{helvet}         
\usepackage{courier}        
\usepackage{type1cm}        
\usepackage{makeidx}         
\usepackage{graphicx}        
\usepackage{multicol}        
\usepackage[bottom]{footmisc}
\usepackage{latexsym,amsmath,amssymb,bm,booktabs,epsfig,multirow,bigstrut,graphicx,url,xspace,natbib}

\makeindex             
                       
\newcommand{\fudge}{\mathfrak{C}}
\spdefaulttheorem{algo}{Algorithm}{\upshape \bfseries}{\upshape}
\newcommand{\me}{\ensuremath{\mathrm{e}}} 

\newcommand\reals{\mathbb{R}}
\newcommand\real{\mathbb{R}}
\newcommand\natu{\mathbb{N}}
\newcommand\e{\mathbb{E}}

\newcommand{\vx}{\boldsymbol{x}}
\newcommand{\bsX}{\boldsymbol{X}}
\newcommand{\vX}{\boldsymbol{X}}
\newcommand{\rd}{\,\mathrm{d}}
\newcommand{\dnorm}{\mathcal{N}}
\newcommand{\Prob}{\Pr}

\newcommand{\abs}[1]{\left|#1\right|}
\DeclareMathOperator{\var}{Var}
\newcommand{\hmu}{\hat{\mu}}

\newcommand{\dif}{\rd}

\newcommand{\hsigma}{\hat{\sigma}}
\newcommand{\talpha}{\tilde{\alpha}}
\newcommand{\tkappa}{\tilde{\kappa}}

\newcommand{\cc}{\mathcal{C}}
\newcommand{\cl}{\mathcal{L}}
\def\abs#1{\ensuremath{\left \lvert #1 \right \rvert}}
\newcommand{\norm}[2][{}]{\ensuremath{\left \lVert #2 \right \rVert}_{#1}}

\begin{document}

\date{\today}
\title*{Guaranteed Conservative Fixed Width Confidence Intervals Via Monte Carlo Sampling
\thanks{The first and second authors were partially supported by the National
Science Foundation under DMS-0923111 and DMS-1115392.  The fourth author was partially supported by the National Science Foundation under DMS-0906056.}}
\titlerunning{Fixed Width Confidence Intervals}
\author{Fred J. Hickernell\inst{1} \and
Lan Jiang\inst{1} \and Yuewei Liu\inst{2} \and Art Owen \inst{3}}
\institute{Department of Applied Mathematics,
Illinois Institute of Technology, Chicago, IL, 60616, USA,
\texttt{hickernell@iit.edu,ljiang14@hawk.iit.edu}
\and
School of Mathematics and Statistics, Lanzhou University, Lanzhou City, Gansu, China 730000, \texttt{lyw@lzu.edu.cn}
\and 
Department of Statistics, Stanford University, Stanford, CA, 94305, USA,
\texttt{owen@stanford.edu}
}
%
%
\maketitle

\abstract{Monte Carlo methods are used to approximate the means, $\mu$, of random variables $Y$, whose distributions are not known explicitly.  The key idea is that the average of a random sample, $Y_1, \ldots, Y_n$, tends to $\mu$ as $n$ tends to infinity. This article explores how one can reliably construct a confidence interval for $\mu$ with a prescribed half-width (or error tolerance) $\varepsilon$.  Our proposed two-stage algorithm assumes that the \emph{kurtosis} of $Y$ does not exceed some user-specified bound. An initial independent and identically distributed (IID) sample is used to confidently estimate the variance of $Y$.  A Berry-Esseen inequality then makes it possible to determine the size of the IID sample required to construct the desired confidence interval for $\mu$.  We discuss the important case where $Y=f(\vX)$ and $\vX$ is a random $d$-vector with probability density function $\rho$.  In this case $\mu$ can be interpreted as the integral $\int_{\reals^d} f(\vx) \rho(\vx) \, \dif \vx$, and the Monte Carlo method becomes a method for multidimensional cubature. }

\section{Introduction}

Monte Carlo algorithms provide a flexible way to approximate $\mu = \e(Y)$ when one can generate samples of the random variable $Y$. For example, $Y$ might be the discounted payoff of some financial derivative, which depends on the future performance of assets that are described by a stochastic model.  Then $\mu$ is the fair option price.  The goal is to obtain a \emph{confidence interval} 
\begin{equation} \label{confint}
\Prob[\abs{\mu - \hmu} \le \varepsilon] \ge 1-\alpha,
\end{equation}
where 
\begin{itemize}

\item $\mu$ is approximated by the sample average of $n$ independent and identically distributed (IID) samples of $Y$,
\begin{equation} \label{eq:samplemean}
\hmu = \hmu_n=\frac 1n \sum_{i=1}^n Y_i,
\end{equation}

\item $\varepsilon$ is the half-width of the confidence interval, which also serves as an \emph{error tolerance}, and

\item $\alpha$ is the level of \emph{uncertainty}, e.g., $1\%$ or $0.1\%$, which is fixed in advance.

\end{itemize}

Often the sample size, $n$, is fixed in advance, and the central limit theorem (CLT) provides an approximate value for $\varepsilon$ in terms of $n$ and 
\begin{equation} \label{sigmadef}
\sigma^2=\var(Y)=\e[(Y-\mu)^2],
\end{equation}
which itself may be approximated by the sample variance.
The goal here is somewhat different.  We want to fix $\varepsilon$ in advance and then determine how large the sample size must be to obtain a fixed
width confidence interval of the form \eqref{confint}. Moreover, we want to make sure that our confidence interval is correct, not just approximately correct, or correct in the limit of vanishing $\varepsilon$.  In this paper we present Algorithm \ref{twostagealgo} for obtaining such a fixed width confidence interval for the mean of a real random
variable when one is performing Monte Carlo sampling.

Before presenting the method, we outline the reasons
that existing fixed width confidence intervals are not suitable.  In summary, there are two drawbacks of existing procedures.  Much existing theory is \emph{asymptotic}, i.e., the proposed procedure attains the desired coverage level in the limit as $\varepsilon\to 0$ but does not provide
coverage guarantees for fixed $\varepsilon>0$.  We want such fixed $\varepsilon$ guarantees.  A second drawback is that the theory may make distributional assumptions that are too strong.  In Monte Carlo applications one typically does not have much information about the underlying distribution.  The form of the distribution for $Y$ is generally not known, $\var(Y)$ is generally not known, and $Y$ is not necessarily bounded. We are aiming to derive fixed width confidence intervals that do not require such assumptions.  

The width (equivalently length) of a confidence interval
tends to become smaller as the number $n$ of sampled
function values increases. In special circumstances, we can choose $n$ to get
a confidence interval of at most the desired length and at
least the desired coverage level, $1-\alpha$. For instance, if the variance, $\sigma^2=\var(Y)$, is known then an approach based on Chebychev's
inequality is available, though the actual coverage
will usually be much higher than the nominal level,
meaning that much narrower intervals would have sufficed.
Known variance in addition to a Gaussian distribution for $Y$
supports a fixed width confidence interval construction that
is not too conservative. The CLT provides a confidence interval that is asymptotically correct, but our aim is for something that is definitely correct for finite sample sizes.
Finally, conservative fixed width confidence intervals
for means can be constructed for bounded random variables, by appealing
to exponential inequalities such as Hoeffding's or Chernoff's inequality.  Unfortunately, $Y$ is often unbounded, e.g., in the case where it represents the payoff of a call option.

If the relevant variance or bound is unknown, then approaches
based on sequential statistics \citep{Sie85a}
may be available.  In sequential methods one keeps increasing
$n$ until the interval is narrow enough. Sequential
confidence intervals require us to take account of the
stopping rule when computing the confidence level. Unfortunately, all existing sequential methods are lacking in some aspects. 

\cite{SerfWack1976} consider sequential confidence intervals for the mean (alternatively for the median) in parametric distributions, symmetric about their center point.  The symmetry condition is not suitable for general
purpose Monte Carlo applications.

\cite{ChoRob65a} develop a sequential sampling fixed width confidence interval procedure for the mean, but its guarantees are only asymptotic (as $\varepsilon \to 0$). \cite{MukhDatt1996} give a procedure similar to Chow and Robbins', and it has similar drawbacks.

Bayesian methods can support a fixed width
interval containing $\mu$ with $1-\alpha$ posterior probability, and
Bayesian methods famously do not require one to account
for stopping rules. They do however require strong distributional assumptions.

There is no assumption-free way to obtain exact confidence intervals for a mean,
as has been known since \cite{BahSav56}. Some kind of assumption is needed to rule out settings where the desired quantity is the
mean of a heavy tailed random variable in which
rarely seen large values dominate the mean and spoil the estimate of the variance.
The assumption we use is an upper bound on
the modified kurtosis (normalized fourth moment) of the
random variable $Y$:
\begin{equation} \label{kurtassump}
\tkappa = \frac{\e[(Y-\mu)^4]}{\sigma^4} \le \tkappa_{\max}.
\end{equation}
(The quantity $\tkappa-3$ is commonly called the kurtosis.)  Under such an assumption we present
a two-stage algorithm: the first stage generates
a conservative upper bound on the variance, and the second stage
uses this variance bound and a Berry-Esseen Theorem, which can be thought of as a non-asymptotic CLT, to determine how large $n$ must be for the sample mean to satisfy confidence interval \eqref{confint}.  Theorem \ref{mainadaptthm} demonstrates the validity of the fixed width confidence interval, and Theorem \ref{costtheorem} demonstrates that the cost of this algorithm is reasonable.  These are our main new theoretical results.

Our procedure is a two-stage procedure
rather than a fully sequential one.  In
this it is similar to the method of
\cite{Stei1945a,Stei1949a}, except that
the latter requires normally distributed
data.

One might question whether assumption \eqref{kurtassump}, which involves fourth
moments of $Y$, is more reasonable than an assumption involving only the second moment of $Y$.  For example, using Chebychev's inequality with the assumption 
\begin{equation} \label{boundedvar}
\sigma^2 \le \sigma^2_{\max}
\end{equation}
also yields a fixed width confidence interval of the form \eqref{confint}.  We would argue that \eqref{kurtassump} is indeed more reasonable.  First, if $Y$ satisfies \eqref{kurtassump}, then so does $cY$ for any nonzero $c$, however, the analog does not hold for \eqref{boundedvar}.  In fact, if $\sigma$ is nonzero, then \eqref{boundedvar} must be violated by $cY$ for $c$ sufficiently large.  Second, making $\tkappa_{\max}$ a factor of $10$ or $100$ larger than $\tkappa$ does not significantly affect the total cost (number of samples required) of our two-stage Monte Carlo Algorithm \ref{twostagealgo} for a large range of values of $\sigma/\varepsilon$.  However, the cost of our Monte Carlo algorithm, and indeed any Monte Carlo algorithm based on IID sampling is proportional to $\sigma^2$, so overestimating $\sigma^2$ by a factor of $10$ or $100$ or more to be safe increases the cost of the algorithm by that factor. 

An important special case of computing $\mu=\e(Y)$ arises in the situation where $Y=f(\vX)$ for some function $f: \reals^d \to \reals$ and some random vector $\vX$ with probability density function $\rho: \reals^d \to [0,\infty)$.  One may then interpret the mean of $Y$ as the multidimensional integral 
\begin{equation} \label{muintegral}
\mu=\mu(f)=\e(Y) = \int_{\reals^d} f(\vx) \rho(\vx) \, \dif \vx.
\end{equation}
Note that unlike the typical probability and statistics setting, where $f$ denotes a probability density function, in this paper $f$ denotes an integrand, and $\rho$ denotes the probability density function.
Given the problem of evaluating $\mu=\int_{\reals^d} g(\vx) \, \dif \vx$, one must choose a probability density function $\rho$ for which one can easily generate random vectors $\vX$, and then set $f = g/\rho$.  The quantities $\sigma^2$ and $\tkappa$ defined above can be written in terms of weighted $\cl_p$-norms of $f$:
\begin{equation} \label{Lpnormdef}
\norm[p]{f} :=\left\{\int_{\reals^d} \abs{f(\vx)}^p \rho(\vx) \, \dif \vx\right\}^{1/p}, \qquad
\sigma^2 = \norm[2]{f-\mu}^2, \qquad \tkappa = \frac{\norm[4]{f-\mu}^4}{\norm[2]{f-\mu}^4}.
\end{equation}
For a given $g$, the choice of $\rho$ is not unique, and making an optimal choice belongs to the realm of \emph{importance sampling}. The assumption of bounded kurtosis, \eqref{kurtassump}, required by Algorithm \ref{twostagealgo}, corresponds to an assumption that the integrand $f$ lies in the \emph{cone} of functions
\begin{equation} \label{conedef}
\cc_{\tkappa_{\max}} = \{ f \in \cl_4 : \norm[4]{f-\mu(f)} \le \tkappa_{\max}^{1/4} \norm[2]{f-\mu(f)} \}.
\end{equation}
This is in contrast to a \emph{ball} of functions, which would be the case if one was satisfying a bounded variance condition, \eqref{boundedvar}.

From the perspective of numerical analysis, if $\rho$ has independent marginals, one may apply a product form of a univariate quadrature rule to evaluate $\mu$.  However, this consumes a geometrically increasing number of samples as $d$ increases, and moreover, such methods often require rather strict smoothness assumptions on $f$.  

If $f$ satisfies moderate smoothness conditions, then (randomized) quasi-Monte Carlo methods, or low discrepancy sampling methods for evaluating $\mu$ are more efficient than simple Monte Carlo \citep{Nie92,SloJoe94,Lem09a,DicPil10a}.  
Unfortunately, practical error estimation remains a challenge for quasi-Monte Carlo methods.  Heuristic methods have been proposed, but they lack theoretical justification. One such heuristic is used with reasonable success in the numerical examples of Section \ref{numerexsec}.  Independent randomizations of quasi-Monte Carlo rules of fixed sample size can be used to estimate their errors, but they do not yet lead to guaranteed, fixed width confidence intervals.

Computational mathematicians have also addressed the problem of constructing automatic algorithms, i.e., given an error tolerance of $\varepsilon$, one computes an approximation, $\hmu$, based on $n$ evaluations of the integrand $f$, such that $\abs{\mu-\hmu} \le \varepsilon$. For example, MATLAB \citep{MAT7.12}, a popular numerical package, contains {\tt quad}, an adaptive Simpson's rule for univariate quadrature routine developed by \cite{GanGau00a}. Although {\tt quad} and other automatic rules generally work well in practice, they do not have any rigorous guarantees that the error tolerance is met, and it is relatively simple to construct functions that fool them.  This is discussed in Section \ref{numerexsec}.  Since a random algorithm, like Monte Carlo, gives a random answer, any statements about satisfying an error criterion must be probabilistic.  This leads us back to the problem of finding a fixed width confidence interval, \eqref{confint}.

An outline of this paper follows. Section \ref{sec:background} defines key terminology and provides certain inequalities used to construct our fixed width confidence intervals.  The new two-stage Algorithm \ref{twostagealgo} is described in Section \ref{sec:twostage}, where rigorous guarantees of its success and its cost are provided.  Section \ref{numerexsec} illustrates the challenges of computing $\mu$ to a guaranteed precision through several numerical examples.  This paper ends with a discussion of our results and further work to be done.

\section{Background probability and statistics}\label{sec:background}

In our Monte Carlo applications, a quantity of interest
is written as an expectation: $\mu = \e(Y)$, where $Y$
is a real valued random variable.  As mentioned above, very often
$Y = f(\bsX)$ where $\bsX\in\real^d$ is a random vector
with probability density function $\rho$. In other settings the random quantity $\bsX$ might
have a discrete distribution or be infinite dimensional (e.g,. a Gaussian
process) or both. For Monte Carlo estimation, we can
work with the distribution of $Y$ alone. The Monte Carlo estimate of $\mu$ is the sample mean, as given in \eqref{eq:samplemean}, where the $Y_i$ are IID random variables with the same distribution as $Y$.

\subsection{Moments}

Our methods require conditions on the first four moments of
$Y$ as described here. The variance of $Y$, as defined in \eqref{sigmadef}, is denoted by $\sigma^2$, and its non-negative square root, $\sigma$, is the standard deviation of $Y$.   Some of our expressions assume without stating it that $\sigma>0$, and all will require $\sigma<\infty$.  The skewness of $Y$ is
$ \gamma = \e[ (Y-\mu)^3 ]/\sigma^3,$
and the kurtosis of $Y$ is
$ \kappa = \tkappa-3 = \e[(Y-\mu)^4]/\sigma^4-3$ (see \eqref{kurtassump}).
The mysterious $3$ in $\kappa$ is there to
make it zero for Gaussian random variables.
Also, $\mu,\sigma^2,\gamma,\kappa$ are related to the
first four cumulants \cite[Chap.\ 2]{McC87a} of the distribution of $Y$, meaning that
\[
\log(\e[\exp(tY)]) = \mu t + \frac{\sigma^2 t^2}{2} + \frac{ \gamma \sigma^3t^3}{3!} + \frac{\kappa \sigma^4 t^4}{4!} + o(t^4).
\]
Our main results require a known upper bound for $\kappa$, which then implies that $\sigma$ and $\gamma$ are finite.

\subsection{CLT intervals}

A random variable $Z$ has the standard normal distribution,
denoted by $\dnorm(0,1)$, if 
$$\Pr( Z\le z ) = \frac1{\sqrt{2\pi}}\int_{-\infty}^z
\exp(-t^2/2)\rd t =: \Phi(z).$$
Under the central limit theorem,
the distribution of $\sqrt{n}(\hat\mu_n-\mu)/\sigma$
approaches $\dnorm(0,1)$ 
as $n\to\infty$, where $\hmu_n$ denotes the sample mean of $n$ IID samples.
As a result
\begin{align}\label{eq:99ci}
\Pr\bigl(
\hat\mu_n-2.58\sigma/\sqrt{n}
\le \mu\le
\hat\mu_n+2.58\sigma/\sqrt{n}\bigr)
\to 0.99
\end{align}
as $n\to\infty$.
We write the interval in~\eqref{eq:99ci}
as $\hat\mu_n\pm 2.58\sigma/\sqrt{n}$.
Equation~\eqref{eq:99ci} cannot be used when
$\sigma^2$ is unknown, but the usual estimate
\begin{align}\label{eq:samplevar}
s^2_n = \frac1{n-1}\sum_{i=1}^n(Y_i-\hat\mu_n)^2
\end{align}
may be substituted, yielding the interval
$\hat\mu_n\pm2.58s_n/\sqrt{n}$ which also
satisfies the limit in~\eqref{eq:99ci}
by Slutsky's theorem \citep{LehRom01a}. For an arbitrary
confidence level $1-\alpha\in(0,1)$, we replace
the constant $2.58$ by $z_{\alpha/2}=\Phi^{-1}(1-\alpha/2)$.
The width of this interval is
$2z_{\alpha/2}s_n/\sqrt{n}$, and when $\mu$ is in
the interval then the absolute error
$|\mu-\hat\mu_n|\le 
\varepsilon := z_{\alpha/2}s_n/\sqrt{n}$.

The coverage level of the CLT interval is only asymptotic. In
more detail, \cite[p.\ 948]{Hal88a} shows that
\begin{align}\label{eq:halls}
\Pr\bigl(
\abs{\mu - 
\hat\mu_n} \le 2.58 s/\sqrt{n}
\bigr)
= 0.99 + \frac1n(A + B\gamma^2+C\kappa) + O\Bigl(\frac1{n^2}\Bigr)
\end{align}
for constants $A$, $B$, and $C$ that depend
on the desired coverage level (here $99$\%).  Hall's theorem
requires only that the random variable $Y$ has sufficiently
many finite moments
and is not supported solely on a lattice (such as
the integers).
It is interesting to note that the $O(1/n)$ coverage error in
\eqref{eq:halls} is better than the $O(1/\sqrt{n})$
root mean squared error for the estimate $\hat\mu_n$ itself.

\subsection{Standard Probability Inequalities}

Here we present some well known inequalities
that we will use.
First, Chebychev's inequality ensures that
a random variable (such as $\hat\mu_n$) is
seldom too far from its mean.

\begin{theorem}[Chebychev's Inequality] \citep[6.1c, p.\ 52]{LinBai10a}\label{Chebineqthm} Let $Z$ be a random variable with mean $\mu$ and variance $\sigma^2 \ge 0$.  
Then for all $\varepsilon >0$,
\[
\Prob[\abs{Z-\mu} \ge \varepsilon ] \le \frac{\sigma^2}{\varepsilon^2}.
\]
\end{theorem}

In some settings we need a one sided inequality
like Chebychev's. We will use this one due to Cantelli.

\begin{theorem}[Cantelli's Inequality] \citep[6.1e, p.\ 53]{LinBai10a} \label{Can}
Let $Z$ be any random variable with mean $\mu$ and finite variance $\sigma^2$.  For any $a\geq 0$, it
follows that:
\[
\Prob[Z-\mu \geq a]\leq \frac{\sigma^2}{a^2+\sigma^2}.
\]
\end{theorem}

Berry-Esseen type theorems govern the rate at which
a CLT takes hold. We will use the following theorem which combines recent work on both uniform and non-uniform ($x$-dependent right hand side) versions.

\begin{theorem}[Berry-Esseen Inequality] \label{BE} Let $Y_1,\dots,Y_n$ be IID random variables with mean
$\mu$, variance $\sigma^2>0$, and third centered moment $M_3=E\abs{Y_i - \mu}^3/\sigma^3 < \infty$.  Let $\hmu_n=(Y_1 +\cdots + Y_n)/n$ denote the sample mean.
Then
\begin{multline*}
\abs{\Prob\left[\frac{\hmu-\mu}{\sigma/\sqrt{n}} <x
\right]-\Phi(x)} \\
\leq \Delta_n(x,M_3) := \frac{1}{\sqrt{n}} \min\left(A_1(M_3+A_2), \frac{A_3 M_3}{1+\abs{x}^{3}} \right) \qquad \forall x \in \reals,
\end{multline*}
where $A_1=0.3328$ and $A_2=0.429$ \citep{She11a}, and $A_3=18.1139$ \citep{NefShe12a}.
\end{theorem}

The constants in the Berry-Esseen Inequality above have been an area of active research.  We would not be surprised if there are further improvements in the near future.

Our method requires probabilistic bounds on the sample variance, $s_n^2$. For that,
we will use some moments of the variance estimate.

\begin{theorem} \cite[Eq.\ (7.16), p.\ 265]{Mil86} \label{Varvarthm} Let $Y_1, \ldots, Y_n$ be IID random variables with variance $\sigma^2$ and modified
kurtosis $\tilde \kappa$ defined in \eqref{kurtassump}.
Let  $s^2_n$ be the sample variance as defined in \eqref{eq:samplevar}.  Then the sample variance is unbiased, $\e(s^2_n)=\sigma^2$, and its variance is
\[
\var(s^2_n) = \frac{\sigma^4}{n} \left ( \tilde\kappa  - \frac{n-3}{n-1} \right).
\]
\end{theorem}

\section{Two-stage confidence interval}\label{sec:twostage}

Our two-stage procedure works as follows.
In the first stage, we take a sample of 
independent values $Y_1,\dots,Y_{n_\sigma}$
from the distribution of $Y$.
From this sample we compute the sample variance, 
$s^2_{n_\sigma}$, according to 
\eqref{eq:samplevar} and estimate
the variance of $Y_i$ by
$\hat\sigma^2 = \fudge^2\hat s_{n_\sigma}^2$, where
$\fudge^2>1$ is a ``variance inflation factor''
that will reduce the probability that we have
underestimated $\sigma^2=\var(Y)$.
For the second stage, we use the estimate
$\hat \sigma^2$ as if it were the true
variance of $Y_i$ and use
Berry-Esseen theorem to obtain a suitable
sample size, $n_{\mu}$, for computing the sample average, $\hmu$, that satisfies the fixed with confidence interval \eqref{confint}.

The next two subsections give details of
these two steps that will let us bound
their error probabilities. Then we give
a theorem on the method as a whole.

\subsection{Conservative variance estimates}

We need to ensure that our first stage estimate of the
variance $\sigma^2$ is not too small. The following
result bounds the probability of such an underestimate.

\begin{lemma}\label{propCant} 
Let $Y_1,\dots,Y_n$ be IID random variables with variance
$\sigma^2>0$ and kurtosis $\kappa$.
Let $s^2_n$ be the sample variance 
defined at \eqref{eq:samplevar}, and let $\tilde\kappa=\kappa+3$.  Then
\begin{subequations} \label{sampvarbd}
\begin{gather}\label{sampvarup}
\Prob\left[s^2_n < \sigma^2 \left\{1 + \sqrt{\left ( \tilde\kappa  - \frac{n-3}{n-1}\right)\left(\frac{1-\alpha}{\alpha n}\right)}\right\} \right] \ge 1 - \alpha, \\
\label{sampvarlo}
\Prob\left[s^2_n > \sigma^2 \left\{1 - \sqrt{\left ( \tilde\kappa  - \frac{n-3}{n-1}\right)\left(\frac{1-\alpha}{\alpha n}\right)}\right\} \right] \ge 1 - \alpha.
\end{gather}
\end{subequations}
\end{lemma}
\begin{proof} Applying Theorem \ref{Varvarthm} and choosing
$$a=\sqrt{\var(s^2_n) \frac{1-\alpha}{\alpha} } = \sigma^2\sqrt{\left(\tkappa-\frac{n-3}{n-1}\right)\left(\frac{1-\alpha}{\alpha n}\right)} >0,
$$
it follows from Cantelli's inequality (Theorem  \ref{Can})  that
\begin{multline*}
\Prob\left[s^2_n-\sigma^2 \geq
\sigma^2\sqrt{\left(\tkappa-\frac{n-3}{n-1}\right)\left(\frac{1-\alpha}{\alpha n}\right)} \right]  = \Prob\left[s^2_n-\sigma^2 \geq
a \right]\\
 \leq \frac{\var(s^2_n)}{a^2+\var(s^2_n)} 
=\frac{\var(s^2_n)}{\var(s^2_n) \frac{1-\alpha}{\alpha}+\var(s^2_n)}  
=\frac{1}{\left(\frac{1-\alpha}{\alpha}\right)+1}=\alpha.
\end{multline*}
Then \eqref{sampvarup} follows directly.  By a similar argument, applying Cantelli's inequality to the expression $\Prob\left[-s^2_n+\sigma^2 \ge a \right]$ implies  \eqref{sampvarlo}. \qed
\end{proof}

Using Lemma~\ref{propCant} we can bound the probability
that $\hat\sigma^2 = \fudge^2 s^2_{n_\sigma}$ overestimates $\sigma^2$.
Equation~\eqref{sampvarup} implies that
\begin{equation*} 
\Prob \left[\frac{s^2_{n_{\sigma}}}{1 - \sqrt{\left ( \tilde\kappa  - \frac{n_{\sigma}-3}{n_{\sigma}-1}\right)\left(\frac{1-\alpha}{\alpha n_{\sigma}}\right)}} > \sigma^2 \right] \ge 1 - \alpha.
\end{equation*}
Thus, it makes sense for us to require the modified kurtosis, $\tkappa$, to be small enough, relative to $n_{\sigma}$, $\alpha$, and $\fudge$, in order to ensure 
that $\Pr(\hat\sigma^2 > \sigma^2) \ge 1-\alpha$. Specifically, we require
$$
\frac{1}{1 - \sqrt{\left ( \tilde\kappa  - \frac{n_{\sigma}-3}{n_{\sigma}-1}\right)\left(\frac{1-\alpha}{\alpha n_{\sigma}}\right)}} \le \fudge^2,
$$
or equivalently,
\begin{equation}
\label{kappamaxdef}
\tilde\kappa \le \frac{n_{\sigma}-3}{n_{\sigma}-1} + \left(\frac{ \alpha n_{\sigma}}{1-\alpha}\right) \left(1 - \frac{1}{\fudge^2}\right)^2 =: \tilde\kappa_{\max} (\alpha,n_{\sigma},\fudge). 
\end{equation}
This condition is the explicit version of \eqref{kurtassump} mentioned in the introduction.

\subsection{Conservative interval widths}

Here we consider how to choose the sample size
$n_\mu$ to get the desired coverage level
from an interval with half-length at most $\varepsilon$.
We suppose here that $\sigma$ is known.  
In practice we will use a conservative (biased high) estimate
for $\sigma$.

First, if the CLT held exactly and not just asymptotically,
then we could use a CLT sample size of
$$
N_{\mathrm{CLT}}(\varepsilon,\sigma,\alpha)
= 
\Bigl\lceil
\Bigl(
\frac{z_{\alpha/2}\sigma}{\varepsilon}
\Bigr)^2
\Bigr\rceil
$$
independent values of $Y_i$ in an interval
like the one in~\eqref{eq:99ci}.

Given knowledge of $\sigma$, but no assurance
of a Gaussian distribution for $\hat\mu_n$, we
could instead select a sample size based on
Chebychev's inequality (Theorem \ref{Chebineqthm}).  Taking
\begin{equation}\label{NCdef}
N_{\text{Cheb}}(\varepsilon,\sigma,\alpha)
= 
\Bigl\lceil\frac{\sigma^2}{\alpha\varepsilon^2}\Bigr\rceil
\end{equation}
IID observations of $Y$ gives the confidence interval \eqref{confint}.
Naturally $N_{\text{Cheb}}\ge N_{\text{CLT}}$.

Finally, we could use the non-uniform Berry-Esseen
inequality from Theorem~\ref{BE}.
This inequality requires a finite scaled third moment
$M_3=E\abs{Y_i - \mu}^3/\sigma^3$.
If $\hmu_n$ denotes a sample mean of $n$ IID random instances of $Y$, then the non-uniform Berry-Esseen inequality implies that
\begin{align} 
\nonumber
\Prob\left[\abs{\mu-\hmu_n}  \le \varepsilon \right]
&=\Prob\left[\frac{\hmu_n - \mu}{\sigma/\sqrt{n}} \le \frac{\sqrt{n}\varepsilon}{\sigma} \right]-\Prob\left[\frac{\hmu_n - \mu}{\sigma/\sqrt{n}} < -\frac{\sqrt{n}\varepsilon}{\sigma}\right]\\ 
\nonumber
&\ge \left[\Phi(\sqrt{n}\varepsilon/\sigma)-\Delta_n(\sqrt{n}\varepsilon/\sigma,M_3)\right] \\
\nonumber
&\qquad \qquad -\left[\Phi(-\sqrt{n}\varepsilon/\sigma) + \Delta_n(-\sqrt{n}\varepsilon/\sigma,M_3)\right]\\
&=1-2[\Phi(-\sqrt{n}\varepsilon/\sigma) + \Delta_n(\sqrt{n}\varepsilon/\sigma,M_3)], \label{BEresult}
\end{align}
since $\Delta_n(-x,M_3)=\Delta_n(x,M_3)$.  The probability of
making an error no greater than $\varepsilon$ is bounded below by $1-\alpha$, i.e., the fixed width confidence interval \eqref{confint} holds with $\hmu=\hmu_n$, provided $n \ge N_{\text{BE}}(\varepsilon,\sigma,\alpha,M_3)$, where the Berry-Esseen sample size is
\begin{equation}\label{NB}
N_{\text{BE}}(\varepsilon,\sigma,\alpha,M_3) := \min \left \{ n \in \natu : \Phi\left(-\sqrt{n}\varepsilon/\sigma  \right)+\Delta_n(\sqrt{n}\varepsilon/\sigma,M_3)
\le \frac{\alpha}{2} \right \}.
\end{equation}
To compute $N_{\text{BE}}(\varepsilon,\sigma,\alpha,M_3)$, we need to know
$M_3$. In practice, substituting an upper
bound on $M_3$ yields an upper
bound on the necessary sample size.

Note that if the $\Delta_n$ term in \eqref{NB} were absent, $N_{\text{BE}}$ would correspond to the CLT sample size $N_{\text{CLT}}$, and in general $N_{\text{BE}}>N_{\text{CLT}}$.   It is possible that in some situations
$N_{\text{BE}}>N_{\text{Cheb}}$ might
hold, and in such cases we could use $N_{\text{Cheb}}$
instead of $N_{\text{BE}}$.

\subsection{Algorithm and Proof of Its Success}

In detail, the two-stage algorithm works
as described below.

\begin{algo}[Two Stage] \label{twostagealgo} The user specifies
four quantities:
\begin{itemize}
\item 
an initial sample size for variance estimation, $n_\sigma \in \{2,3, \ldots\}$,
\item
a variance inflation factor $\fudge^2\in(1,\infty)$,
\item
an uncertainty $\alpha\in(0,1)$, and,
\item
an error tolerance or confidence interval half-width, $\varepsilon>0$.
\end{itemize}

At the first stage of the algorithm,
$Y_1,\dots,Y_{n_\sigma}$ are sampled independently
from the same distribution as $Y$.
Then the conservative variance estimate, $\hat\sigma^2 = \fudge^2 s^2_{n_\sigma}$,
is computed in terms of the sample variance, $s^2_{n_\sigma}$, defined by \eqref{eq:samplevar}.

To prepare for the second stage of the algorithm
we compute $\tilde\alpha = 1-\sqrt{1-\alpha}$
and then $\tilde\kappa_{\max} = \tilde\kappa_{\max}(\tilde\alpha,n_\sigma,\fudge)$
using equation~\eqref{kappamaxdef}.
The sample size for the second stage is
\begin{equation} \label{nmudef}
n_\mu = N_{\mu}(\varepsilon,\hsigma,\tilde\alpha,\tilde\kappa_{\max}^{3/4}),
\end{equation}
where
\begin{equation} \label{NCBdef}
N_{\mu}(\varepsilon,\sigma,\alpha,M) 
:= \max\bigl(1,\min\bigl(N_{\text{Cheb}}(\varepsilon,\sigma,\alpha), 
N_{\text{BE}}(\varepsilon,\sigma,\alpha,M) \bigr) \bigr).
\end{equation} 
Recall that
$N_{\text{Cheb}}$ is defined in \eqref{NCdef} and  $N_{\text{BE}}$ 
is defined in \eqref{NB}.  

After this preparation, the second stage is to sample
$Y_{n_\sigma+1},\dots,Y_{n_\sigma+n_\mu}$ independently
from the distribution of $Y$, and independently of $Y_{1},\dots,Y_{n_\sigma}$.  The algorithm then returns the sample mean,
\begin{align}\label{eq:theestimate}
\hmu = \frac1{n_\mu}\sum_{i=n_\sigma+1}^{n_\sigma+n_\mu}Y_i.
\end{align}
\end{algo}

\bigskip

The success of this algorithm is guaranteed in the following theorem.  The main assumption needed is an upper bound on the kurtosis.

\begin{theorem} \label{mainadaptthm} 
Let $Y$ be a random variable with mean $\mu$, and either zero variance or positive variance with modified kurtosis $\tkappa \le \tilde\kappa_{\max}(\tilde\alpha,n_\sigma,\fudge)$.  It follows that Algorithm \ref{twostagealgo} above yields an estimate
$\hat\mu$ given by~\eqref{eq:theestimate} which satisfies the fixed width confidence interval condition
$$\Pr( |\hat\mu-\mu|\le\varepsilon)\ge 1-\alpha.$$
\end{theorem}
\begin{proof}
\smartqed  
If $\sigma^2=0$, then $s_{n_\sigma}^2=0$, $n_\mu=1$ and $\hmu=\mu$ with probability one.  Now consider the case of positive variance.
The first stage yields a variance estimate satisfying
$
\Pr( \hsigma^2 >\sigma^2)\ge 1-\tilde\alpha
$
by the argument preceding the kurtosis bound in \eqref{kappamaxdef} applied with uncertainty $\tilde\alpha$.
The second stage yields
$\Pr( |\hat\mu-\mu|\le\varepsilon)\ge 1-\tilde\alpha$
by the Berry-Esseen result~\eqref{BEresult},
so long as $\hat\sigma\ge\sigma$
and $M_3\le \tilde\kappa_{\max}(\tilde\alpha,n_\sigma,\fudge)^{3/4}$.
The second condition holds because $M_3 \le \tkappa^{3/4}$ by Jensen's Inequality \citep[8.4.b]{LinBai10a}.
Thus, in the two-stage algorithm we have
\begin{align*}
\Prob\left(\abs{\hmu-\mu} \le \varepsilon \right) &
= \e\bigl[\Prob\left(\abs{\hmu-\mu} \le \varepsilon \mid \hsigma \right) \bigr] \\
& \ge \e\left[(1-\tilde\alpha) 1_{\sigma\le\hsigma}\right]\\
& \ge (1-\tilde\alpha) (1-\tilde\alpha) = 1-\alpha.\qquad \qquad \qed
\end{align*}
\end{proof}

\begin{remark} As pointed out earlier, the guarantees in this theorem require that the 
modified kurtosis of $Y$ not exceed the specified upper bound $\tkappa_{\max}$.  As it is presented, Algorithm \ref{twostagealgo} takes as inputs, $n_\sigma$, $\fudge$, and $\alpha$, and uses these to compute  $\tkappa_{\max}$ according to \eqref{kappamaxdef}.  The reason for doing so is that one might have a better intuition for $n_\sigma$, $\fudge$, and $\alpha$.  Alternatively, one may specify $n_\sigma$ and $\tkappa_{\max}$ and use \eqref{kappamaxdef} to compute $\fudge$, or specify $\fudge$ and $\tkappa_{\max}$ and use \eqref{kappamaxdef} to compute $n_\sigma$.  
The issue of how one should choose $n_\sigma$, $\fudge$, and $\tkappa_{\max}$ in practice is discussed further in Section \ref{discusssec}.
\end{remark}

\begin{remark} In this algorithm it is possible to choose $n_{\mu}$ much smaller than  $n_\sigma$ if the sample variance is small.  As a practical matter we suggest that if one is willing to invest $n_\sigma$ samples to estimate the variance then one should be willing to invest at least that many additional samples to estimate the mean.  Therefore, in the numerical examples of Section \ref{numerexsec} we use
\begin{equation} \label{NCBpracticaldef}
N_{\mu}(\varepsilon,\sigma,\alpha,M) 
:= \max\bigl( n_{\sigma}, \min\bigl(N_{\text{Cheb}}(\varepsilon,\sigma,\alpha), 
N_{\text{BE}}(\varepsilon,\sigma,\alpha,M) \bigr) \bigr)
\end{equation} 
instead of \eqref{NCBdef} to determine the sample size for the sample mean.  Because the variance is typically harder to estimate accurately than the mean, one may wonder whether $n_\sigma$ should be chosen greater than $n_\mu$.  However, for Monte Carlo simulation we only need the variance to one or two digits accuracy, whereas we typically want to know the mean to a much higher accuracy.  By the error bound following from Chebychev's inequality (Theorem \ref{Chebineqthm}), the definition of $N_{\mu}$ in \eqref{NCBpracticaldef} means that the fixed width confidence interval constructed by Algorithm \ref{twostagealgo} also holds for any random variables, $Y$, with small variance, namely, $\sigma^2 \le \varepsilon^2 \alpha n_{\sigma}$, even if its kurtosis is arbitrarily large.
\end{remark}

\bigskip

As mentioned in the introduction, one frequently encountered case occurs when  $Y$ is a $d$-variate function of a random vector $\vX$. Then $\mu$ corresponds to the multivariate integral in \eqref{muintegral} and Theorem \ref{mainadaptthm} may be interpreted as below:

\begin{corollary} \label{integcor} Suppose that $\rho:\reals^d \to \reals$ is a probability density function, the integrand $f: \reals^d \to \reals$ has finite $\cl_4$ norm as defined in \eqref{Lpnormdef}, and furthermore $f$ lies in the cone $\cc_{\tkappa_{\max}}$ defined in \eqref{conedef}, where $\tkappa_{\max}=\tilde\kappa_{\max}(\tilde\alpha,n_\sigma,\fudge)$.  It follows that Algorithm  \ref{twostagealgo}  yields an estimate,
$\hat\mu$, of the multidimensional integral $\mu$ defined in \eqref{muintegral}, which satisfies the fixed width confidence interval condition
$$\Pr( |\hat\mu-\mu|\le\varepsilon)\ge 1-\alpha.$$
\end{corollary}

\subsection{Cost of the Algorithm} \label{algcostsubsec} The number of function values required by the two-stage Algorithm \ref{twostagealgo} is $n_{\sigma}+n_\mu$, the sum of the initial sample size used to estimate the variance of $Y$ and the sample size used to estimate the mean of $Y$.  Although $n_{\sigma}$ is deterministic, $n_\mu$ is a random variable, and so the cost of this algorithm might be best defined probabilistically.  Moreover, the only random quantity in the formula for $n_{\mu}$ in \eqref{nmudef} is $\hsigma^2$, the upper bound on variance.  Clearly this depends on the unknown population variance, $\sigma^2$, and we expect $\hsigma^2$ not to overestimate $\sigma^2$ by much.  Thus, the algorithm cost is defined below in terms of $\sigma^2$ and the error tolerance (interval half-width) $\varepsilon$.  An upper bound on the cost is then derived in Theorem \ref{costtheorem}.

Let $A$ be any random algorithm that takes as its input, a method for generating random samples, $Y_1, Y_2, \ldots $ with common distribution function $F$ having variance $\sigma^2$ and modified kurtosis $\tkappa$.  Additional algorithm inputs are an error tolerance, $\varepsilon$, an uncertainty, $\alpha$, and a maximum modified kurtosis, $\tkappa_{\max}$.  The algorithm then computes $\hmu=A(F,\varepsilon,\alpha,\tkappa_{\max})$, an approximation to $\mu=\e(Y)$, based on a total of $N_{\text{tot}}(\varepsilon,\alpha,\tkappa_{\max},F)$ samples. The probabilistic cost of the algorithm, with uncertainty $\beta$, for integrands of variance no greater than $\sigma^2_{\max}$ and modified kurtosis no greater than $\tkappa_{\max}$ is defined as 
\begin{equation*}
N_{\text{tot}}(\varepsilon,\alpha,\beta,\tkappa_{\max},\sigma_{\max})
:= \sup_{\substack{\tkappa \le \tkappa_{\max} \\ \sigma \le \sigma_{\max}} } \min\left\{N : \Prob[N_{\text{tot}}(\varepsilon,\alpha,\tkappa_{\max},F) \le N] \ge 1-\beta  \right \}.
\end{equation*}
Note that $\tkappa_{\max}$ is an input to the algorithm, but $\sigma_{\max}$ is not.  The cost of an arbitrary algorithm, $A$ may also depend on other parameters, such as $n_\sigma$ and $\fudge$ in our Algorithm \ref{twostagealgo}, which are related to $\tkappa_{\max}$.  However, this dependence is not shown explicitly to keep the notation simple.

The cost of the particular two-stage Monte Carlo algorithm defined in Algorithm \ref{twostagealgo} is
\begin{equation*}
\sup_{\substack{\tkappa \le \tkappa_{\max} \\ \sigma \le \sigma_{\max}}} \min\left\{N : \Prob(n_{\sigma} + N_{\mu}(\varepsilon,\hsigma,\tilde\alpha,\tilde\kappa_{\max}^{3/4}) \le N) \ge 1-\beta  \right \}.
\end{equation*}
Since $n_{\sigma}$ is fixed, bounding this cost depends on bounding $N_{\mu}(\varepsilon,\hsigma,\tilde\alpha,\tilde\kappa_{\max}^{3/4})$, which depends on $\hsigma$ as given by Algorithm \ref{twostagealgo}.  Moreover, $\hsigma$ can be bounded above using \eqref{sampvarup} in Lemma~\ref{propCant}.  For $\tkappa \le \tkappa_{\max}$, 
\begin{align*}
1-\beta & \le \Prob\left[s^2_{n_{\sigma}} < \sigma^2 \left\{1 + \sqrt{\left ( \tkappa  - \frac{n_{\sigma}-3}{n_{\sigma}-1}\right)\left(\frac{1-\beta}{\beta n_{\sigma}}\right)}\right\} \right] \\
& \le \Prob\left[\hsigma^2 = \fudge^2 s^2_{n_{\sigma}} < \fudge^2\sigma^2 \left\{1 + \sqrt{\left ( \tkappa_{\max}(n_{\sigma},\talpha,\fudge)  - \frac{n_{\sigma}-3}{n_{\sigma}-1}\right)\left(\frac{1-\beta}{\beta n_{\sigma}}\right)}\right\} \right] \\
& = \Prob\left[\hsigma^2 < \sigma^2 v^2(\talpha,\beta,\fudge) \right],
\end{align*}
where
\[
v^2(\talpha,\beta,\fudge) :=  \fudge^2 + \left(\fudge^2 - 1\right)\sqrt{\frac{ \talpha(1-\beta)}{(1-\talpha)\beta} } > 1.
\]
Noting that $N_{\mu}(\varepsilon,\cdot,\tilde\alpha,\tilde\kappa_{\max}^{3/4})$ is a non-decreasing function allows one to derive the following upper bound on the cost of the adaptive Monte Carlo algorithm.

\begin{theorem} \label{costtheorem} The two-stage Monte Carlo algorithm for fixed width confidence intervals based on IID sampling described in Algorithm \ref{twostagealgo} has a probabilistic cost bounded above by 
\begin{multline*}
N_{\rm{tot}}(\varepsilon,\alpha, \beta, \tkappa_{\max}, \sigma_{\max}) \\
\le
N_{\rm{up}}(\varepsilon,\alpha, \beta, \tkappa_{\max}, \sigma_{\max}) :=  n_{\sigma} + N_{\mu}(\varepsilon,\sigma_{\max}v(\talpha,\beta,\fudge),\talpha,\tkappa_{\max}^{3/4}) .
\end{multline*} 
\end{theorem}

Note that the Chebychev sample size, $N_{\text{Cheb}}$, defined in \eqref{NCdef}, the Berry-Esseen sample size, $N_{\text{BE}}$, defined in \eqref{NB}, and thus $N_\mu$ all depend on $\sigma$ and $\varepsilon$ through their ratio, $\sigma/\varepsilon$.  Thus, ignoring the initial sample used to estimate the variance, $N_{\rm{tot}}(\varepsilon,\alpha, \beta,\tkappa_{\max}, \sigma_{\max})$ is roughly proportional to $\sigma^2_{\max}/\varepsilon^{2}$, even though $\sigma_{\max}$ is not a parameter of the algorithm.  Algorithm \ref{twostagealgo} \emph{adaptively} determines the sample size, and thus the cost, to fit the unknown variance of $Y$. Random variables, $Y$, with small variances will require a lower cost to estimate $\mu$ with a given error tolerance than random variables with large variances.

\begin{figure}
\centering
\begin{minipage}{2.1in}
\centering \includegraphics[width=2.1in]{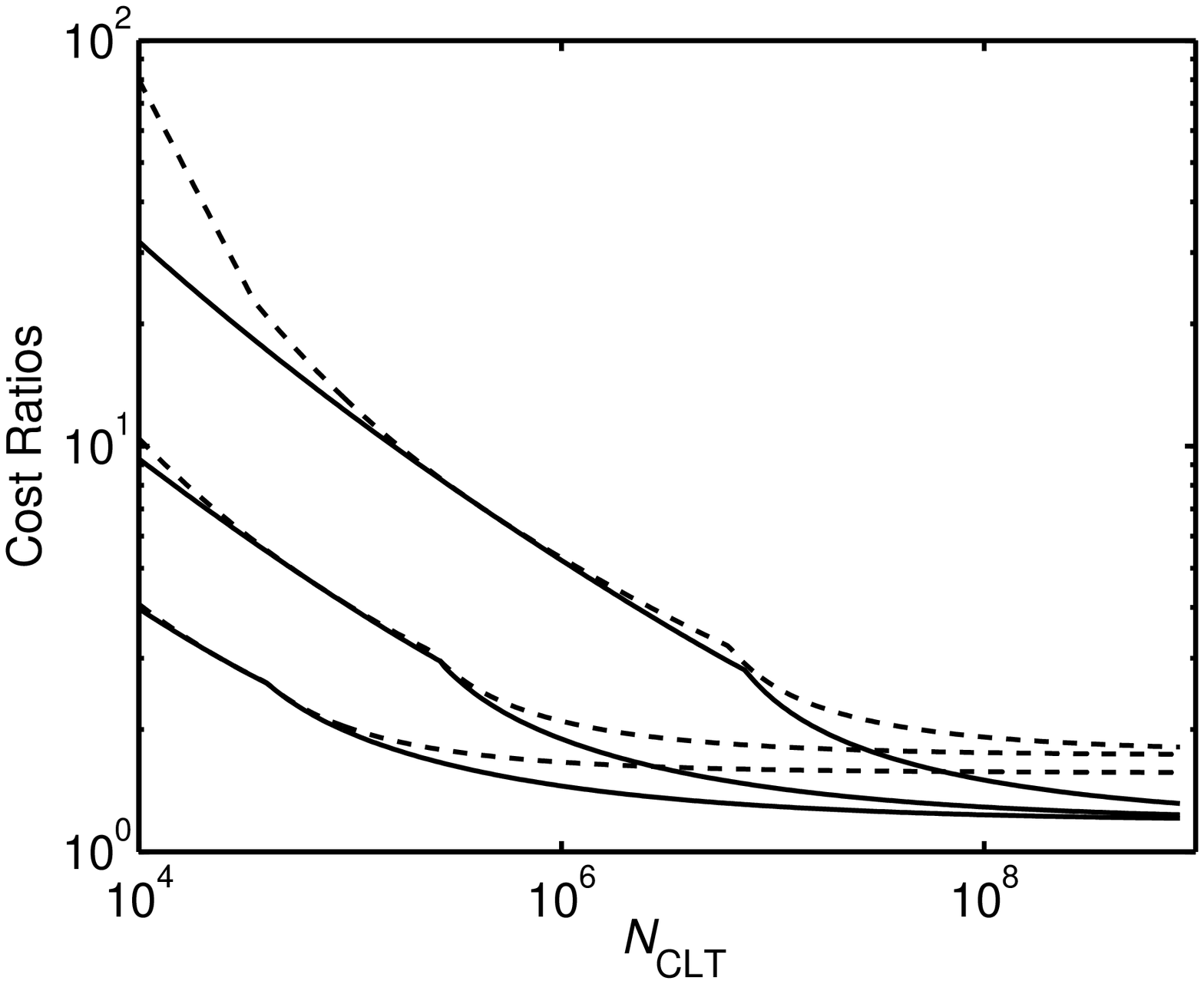} \\
(a)
\end{minipage}
\quad 
\begin{minipage}{2.3in}\centering
\includegraphics[width=2.3in]{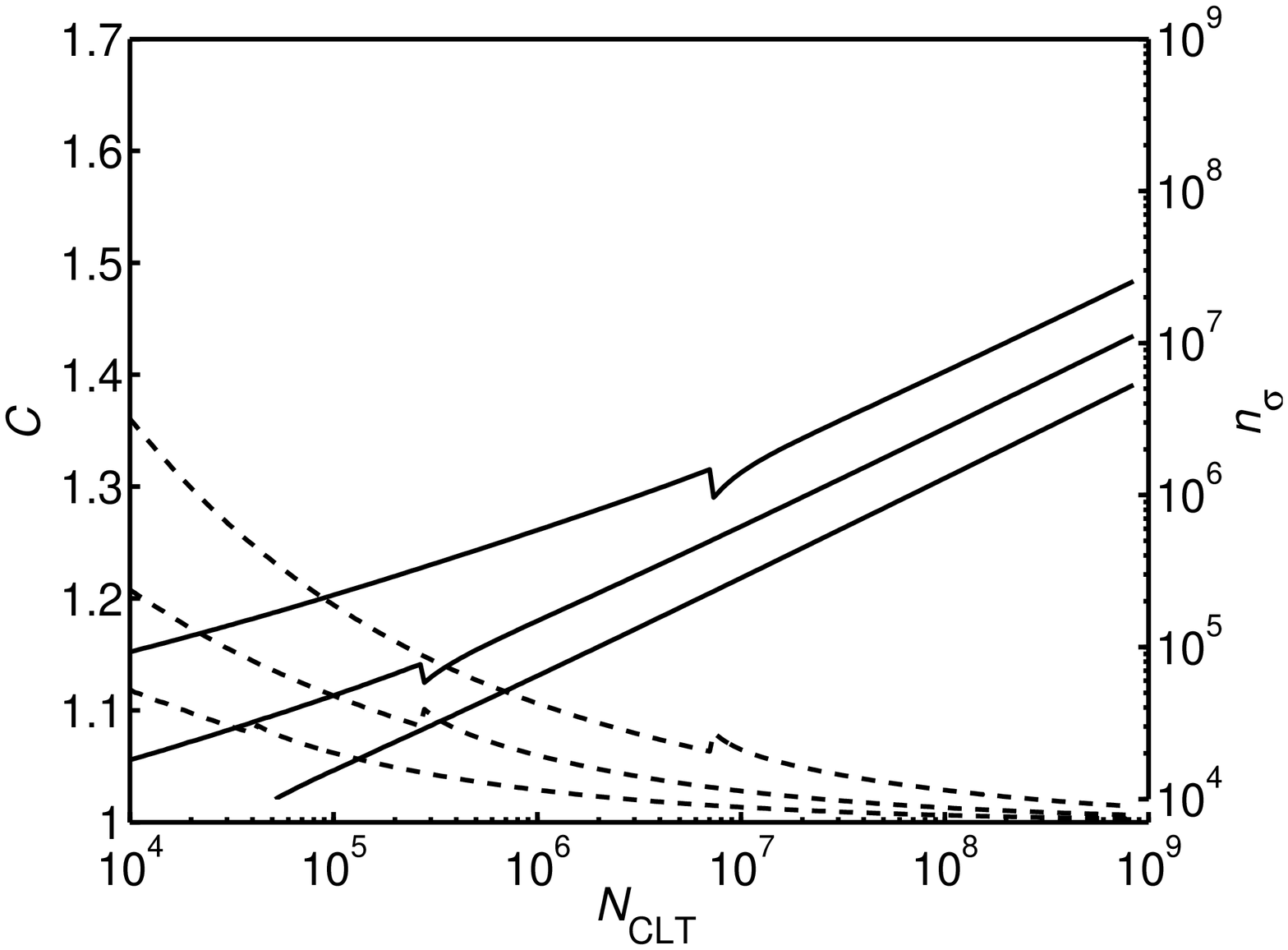}\\
(b)
\end{minipage}
\caption{(a) The cost ratios of $N_{\rm{up}}(\varepsilon,0.01,0.01, \tkappa_{\max}, \sigma)/N_{\mathrm{CLT}}(\varepsilon,\sigma,0.01)$ for $\tkappa_{\max}=2, 10,$ and $100$, with $n_\sigma=4000\tkappa_{\max}$ (dashed) and $n_\sigma$ optimized (solid); (b) the optimal values of $n_\sigma$ (solid) and $\fudge$ (dashed).\label{Costfig}}
\end{figure}

Figure \ref{Costfig}a shows the ratio of the upper bound of the cost, $N_{\rm{up}}(\varepsilon,0.01,0.01, \tkappa_{\max}, \sigma)$, to the ideal CLT cost, $N_{\mathrm{CLT}}(\varepsilon,\sigma,0.01)=\lceil (2.58\sigma/\varepsilon)^2\rceil$, for a range of $\sigma/\varepsilon$ ratios and for $\tkappa_{\max}=2, 10$, and $100$. In these graphs the formula defining $N_{\rm{up}}$ in Theorem \ref{costtheorem} uses the alternative and somewhat costlier formula for $N_{\mu}$ in \eqref{NCBpracticaldef}. The dashed curves in Figure \ref{Costfig}a show these cost ratios with $n_\sigma=4000\tkappa_{\max}$, which corresponds to $\fudge \approx 1.1$.  The solid curves denote the case where $n_\sigma$ and $\fudge$ vary with $\sigma/\varepsilon$ to minimize $N_{\rm{up}}$.  Figure \ref{Costfig}b displays the optimal values of $n_\sigma$ (solid) and $\fudge$ (dashed).  In both figures, higher curves correspond to higher values of $\tkappa_{\max}$. 

Here, $N_{\mathrm{CLT}}$ denotes the ideal cost if one knew the variance of $Y$ a priori and knew that the distribution of the sample mean was close to Gaussian. The cost ratio is the penalty for having a guaranteed fixed width confidence interval in the absence of this knowledge about the distribution of $Y$.  For smaller values of $N_{\mathrm{CLT}}$, equivalently smaller  $\sigma/\varepsilon$, this cost ratio can be rather large.  However the absolute effect of this large penalty is mitigated by the fact that the total number of samples needed is not much.  For larger $N_{\mathrm{CLT}}$, equivalently larger  $\sigma/\varepsilon$, the cost ratio approaches somewhat less than $1.4$ in the case of optimal $n_\sigma$ and $\fudge$, and somewhat less than $2$ for $n_\sigma=1000\tkappa_{\max}$.

The discontinuous derivatives in the curves in Figure \ref{Costfig} arise from the minimum and maximum values arising in formulas \eqref{NB} and \eqref{NCBpracticaldef} for $N_{\text{BE}}$ and $N_\mu$, respectively. Taking the upper dashed curve in Figure \ref{Costfig}a as an example, for $N_{\text{CLT}}$ less than about $3.5\times 10^4$, $N_\mu=n_\sigma$.  For $N_{\text{CLT}}$ from about $3.5\times 10^4$ to about $6\times 10^6$, $N_\mu$ corresponds to the second term in the minimum in the Berry-Esseen inequality, \eqref{NB}, i.e., the non-uniform term.  For $N_{\text{CLT}}$ greater than $6\times 10^6$, $N_\mu$ corresponds to the first term in the minimum in the Berry-Esseen inequality, \eqref{NB}, i.e., the uniform term.

The ideal case of optimizing $n_\sigma$ and $\fudge$ with respect to $\sigma/\varepsilon$ is impractical, since $\sigma$ is not known in advance.  Our suggestion is to choose $\fudge$ around $1.1$, and then choose $n_\sigma$ as large as needed to ensure that $\tkappa_{\max}$ is as large as desired. For example with  $\fudge = 1.1$ and 
$\tkappa_{\max} = 2, 10$, and $100$
we get $n_\sigma = 6593$, $59311,$ and $652417$ respectively.

\section{Numerical Examples} \label{numerexsec}

\subsection{Univariate Fooling Functions for Deterministic Algorithms}

Several commonly used software packages have automatic algorithms for integrating functions of a single variable.  These include 
\begin{itemize} 

\item {\tt quad} in MATLAB \citep{MAT7.12}, adaptive Simpson's rule based on {\tt adaptsim} by \cite{GanGau00a},

\item {\tt quadgk} in MATLAB \citep{MAT7.12}, adaptive Gauss-Kronrod quadrature based on {\tt quadva} by \cite{Sha08a}, and

\item the {\tt chebfun} \citep{TrefEtal12} toolbox for MATLAB \citep{MAT7.12}, which approximates integrals by integrating interpolatory Chebychev polynomial approximations to the integrands.


\end{itemize}

For these three automatic algorithms one can easily probe where they sample the integrand, feed the algorithms zero values, and then construct fooling functions for which the automatic algorithms will return a zero value for the integral.  Figure \ref{foolfunfig} displays these fooling functions for the problem $\mu=\int_0^1 f(x) \, \dif x$ for these three algorithms. Each of these algorithms is asked to provide an answer with an absolute error no greater than $10^{-14}$, but in fact the absolute error is $1$ for these fooling functions.  The algorithms {\tt quad} and {\tt chebfun} sample only about a dozen points before concluding that the function is zero, whereas the algorithm {\tt quadgk} samples a much larger number of points (only those between $0$ and $0.01$ are shown in the plot). 

\begin{figure}
\centering
\begin{minipage}{3.7cm} \centering \includegraphics[width=3.7cm]{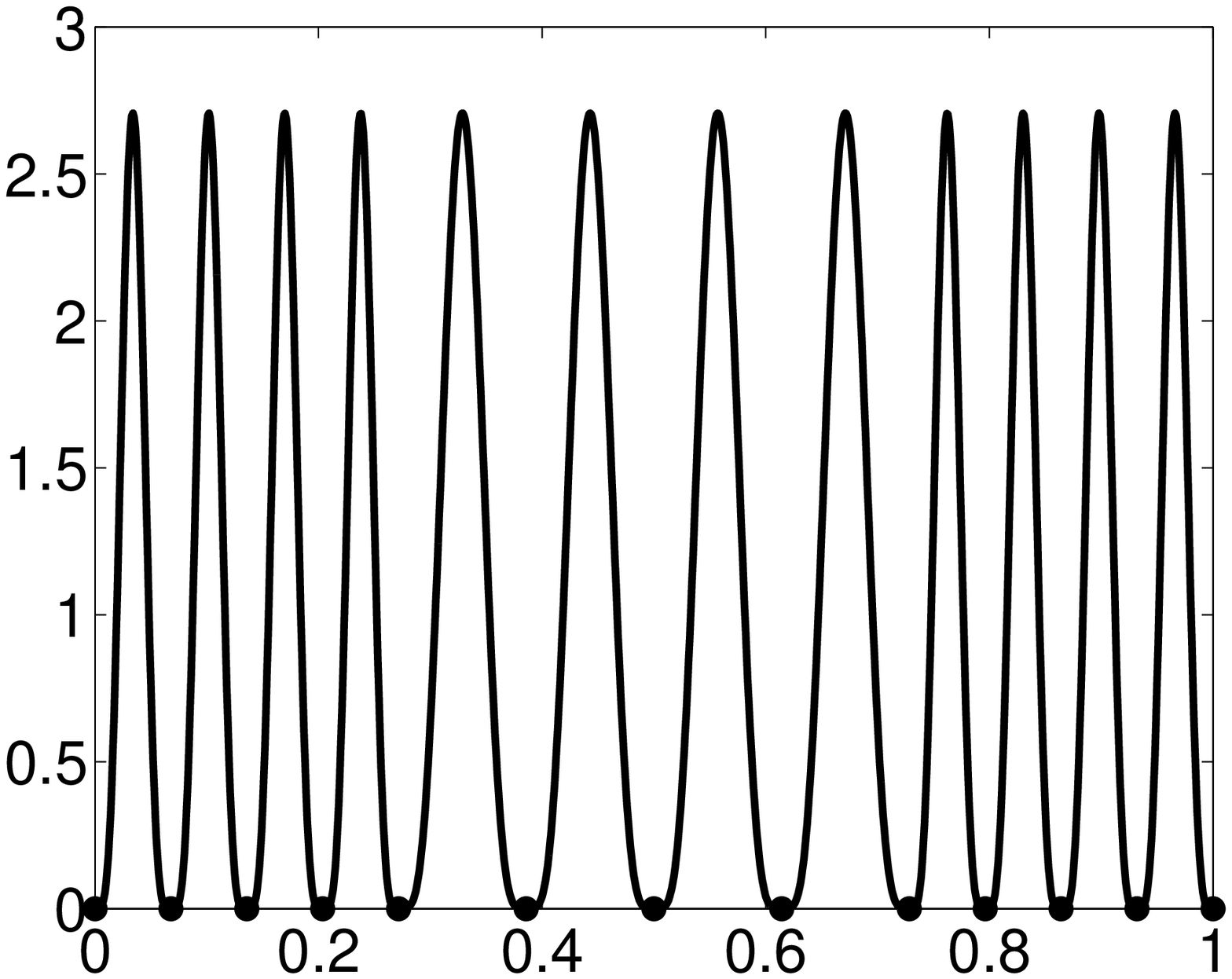} \\ {\tt quad} \end{minipage}
\begin{minipage}{3.7cm} \centering \includegraphics[width=3.7cm]{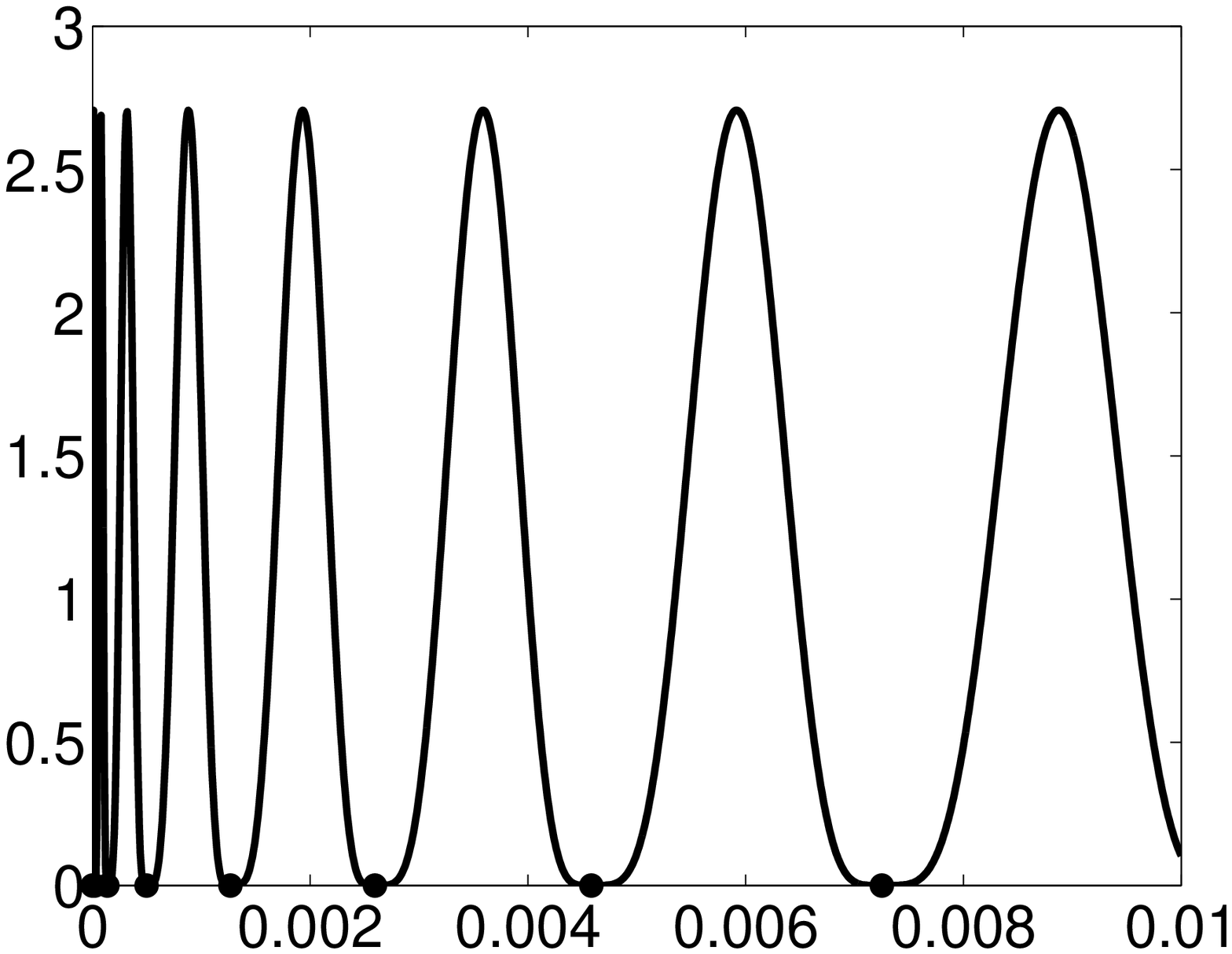} \\ {\tt quadgk} \end{minipage}
\begin{minipage}{3.7cm} \centering \includegraphics[width=3.7cm]{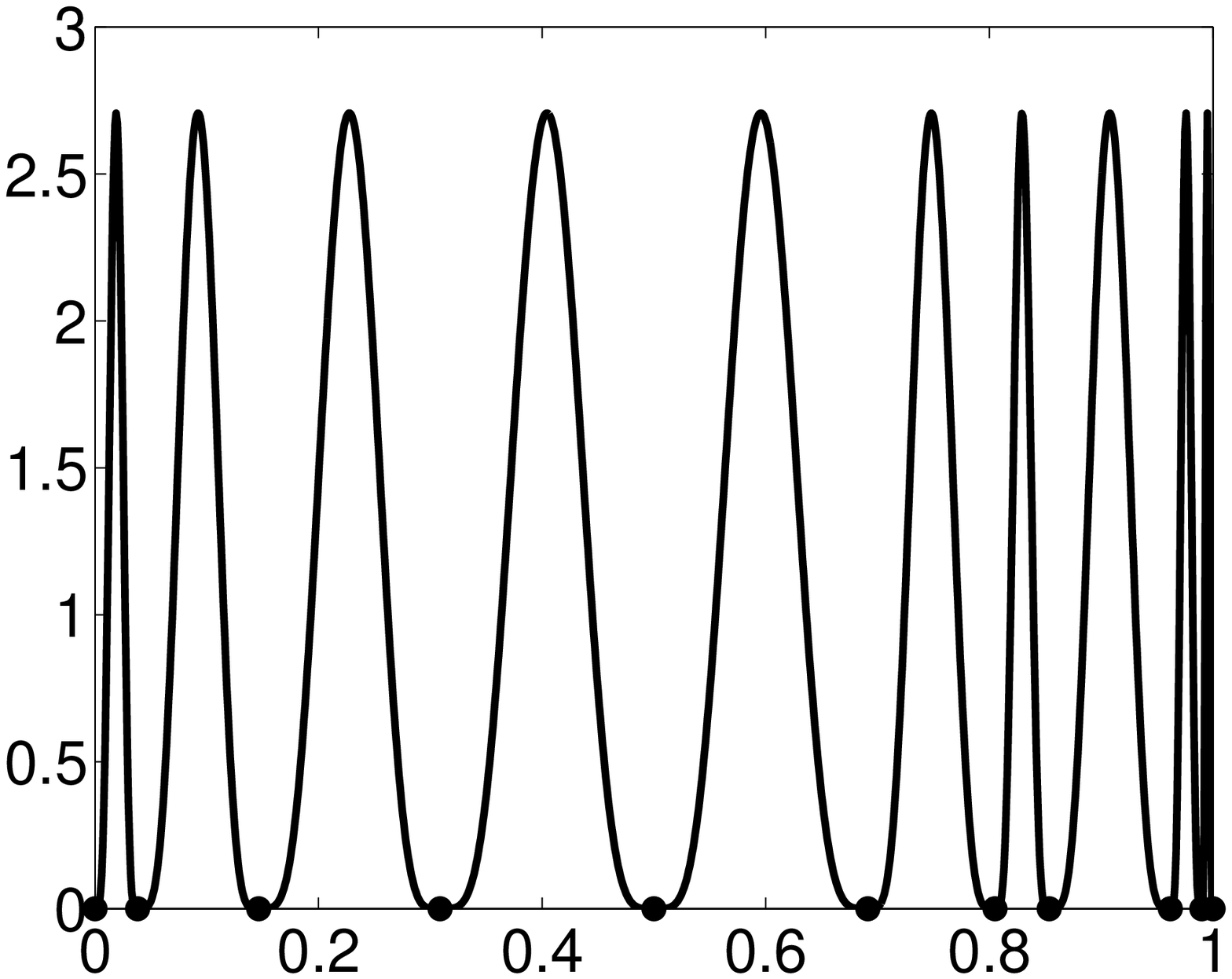} \\ {\tt chebfun} \end{minipage}
\caption{Plots of fooling functions, $f$, with $\mu=\int_0^1 f(x) \, \dif x=1$, but for which the corresponding algorithms return values of $\hmu=0$. \label{foolfunfig}}
\end{figure}

\subsection{Integrating a Single Hump}

Accuracy and timing results have been recorded for the integration problem $\mu=\int_{[0,1]^d} f(\vx) \, \dif \vx$ for a single hump test integrand
\begin{equation} \label{GaussianTestFun}
f(\vx)=a_0 + b_0\prod_{j=1}^d\left[ 1 +b_j \exp \left(-\frac{(x_j-h_j)^2}{c_j^2}\right) \right].
\end{equation}
Here $\vx$ is a $d$ dimensional vector, and $a_0, b_0, \ldots, b_d, c_1, \ldots, c_d, h_1, \ldots, h_d$ are parameters. Figure \ref{GaussianTestFunFig} shows the results of different algorithms being used to integrate $500$ different instances of $f$.  For each instance of $f$, the parameters are chosen as follows:
\begin{itemize} 
\item $b_1, \ldots, b_d \in [0.1,10]$ with $\log(b_j)$ being i.i.d.\ uniform,
\item $c_1, \ldots, c_d \in [10^{-6},1]$ with $\log(c_j)$ being i.i.d.\ uniform,
\item $h_1, \ldots, h_d \in [0,1]$ with $h_j$ being i.i.d.\ uniform,
\item $b_0$ chosen in terms of the $b_1, \ldots, b_d, c_1, \ldots, c_d, h_1, \ldots, h_d$ to make $\sigma^2 =\norm[2]{f-\mu}^2 \in [10^{-2}, 10^2]$, with $\log(\sigma)$ being i.i.d.\ uniform for each instance, and
\item $a_0$ chosen in terms of the $b_0, \ldots, b_d, c_1, \ldots, c_d, h_1, \ldots, h_d$ to make $\mu=1$.
\end{itemize}
These ranges of parameters are chosen so that the algorithms being tested fail to meet the error tolerance a significant number of times.

\begin{figure}
\centering
\begin{minipage}{5.7cm} \centering \includegraphics[width=5.7cm]{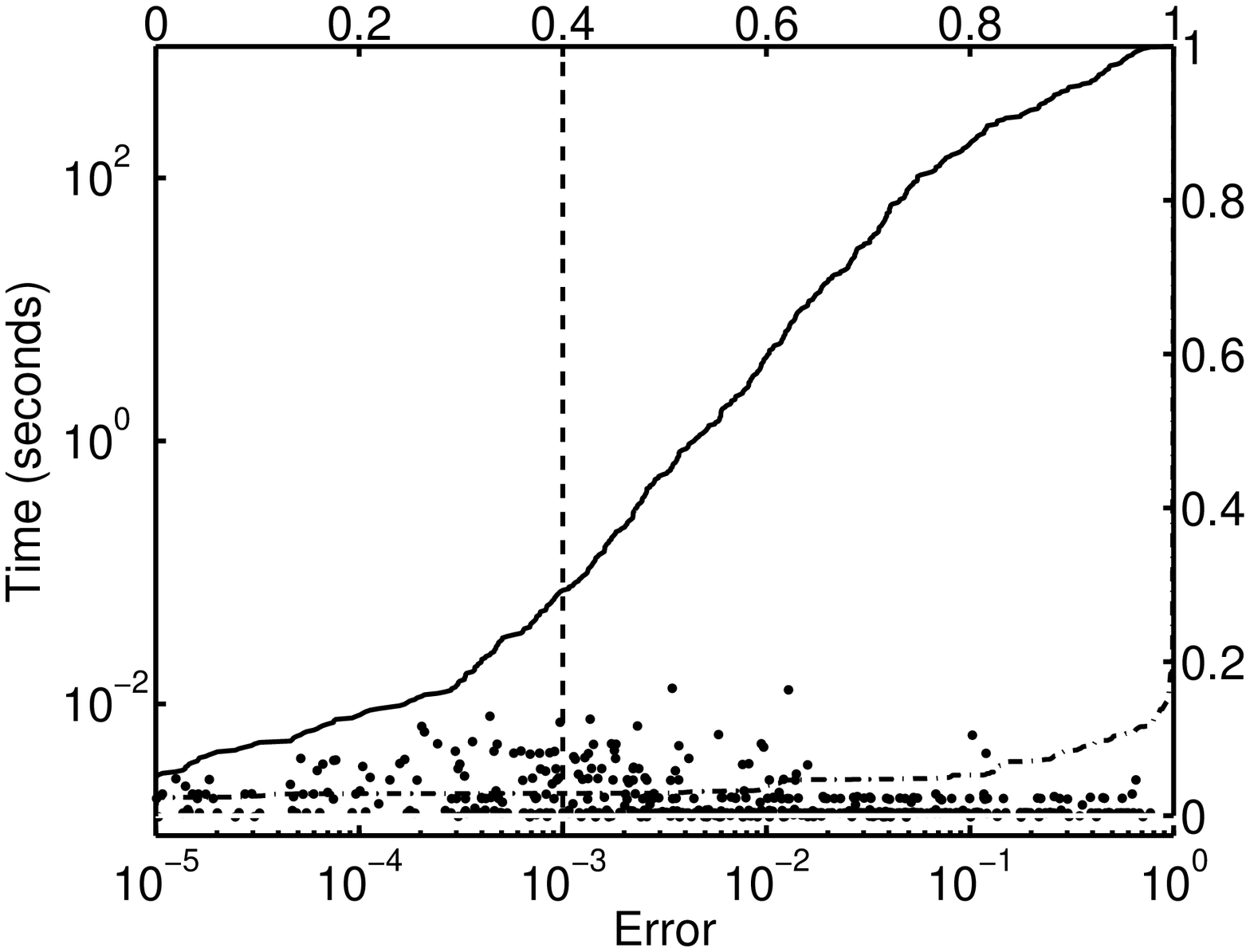} \\ {\tt quad} \end{minipage}
\begin{minipage}{5.7cm} \centering \includegraphics[width=5.7cm]{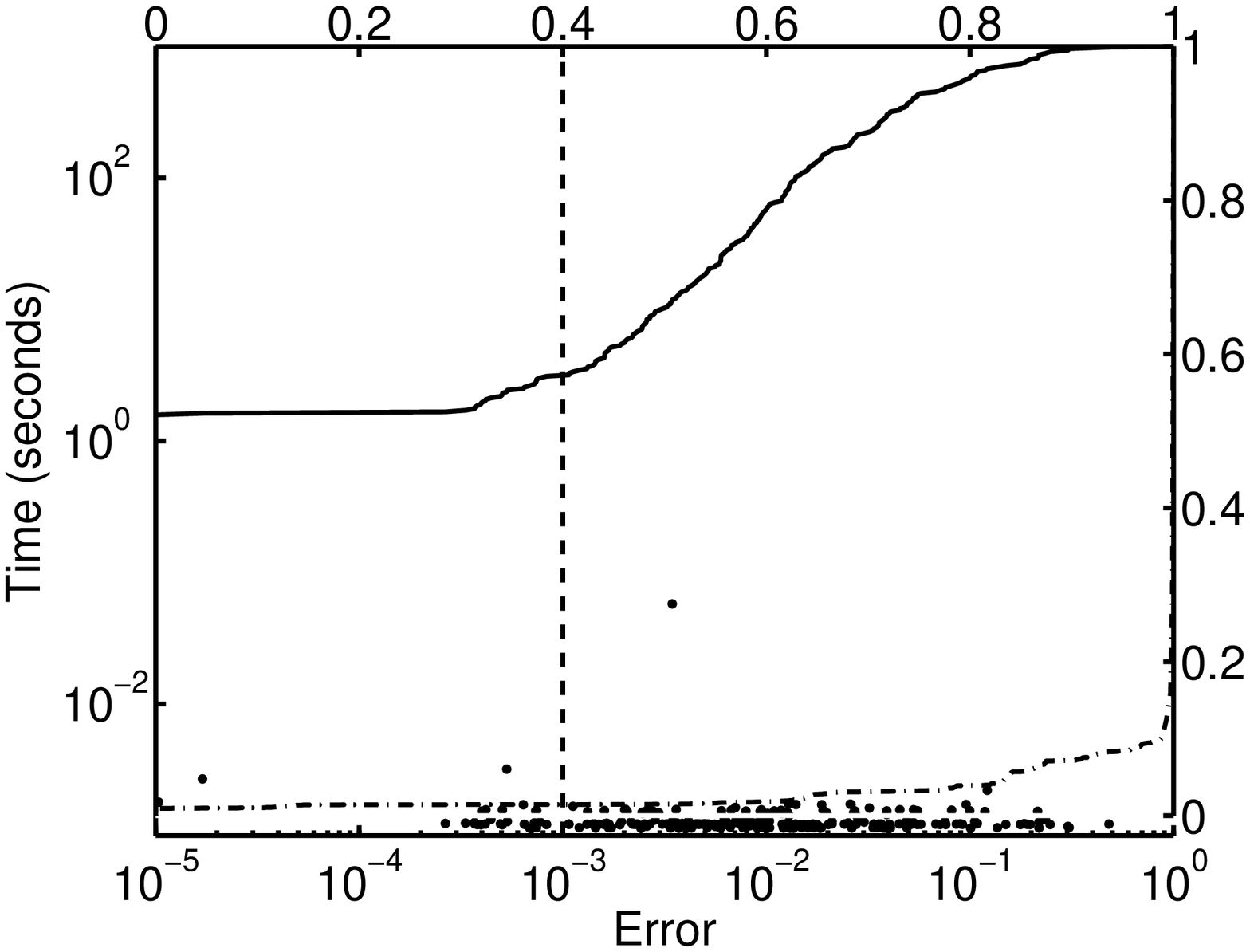} \\ {\tt quadgk} \end{minipage}
\begin{minipage}{5.7cm} \centering \includegraphics[width=5.7cm]{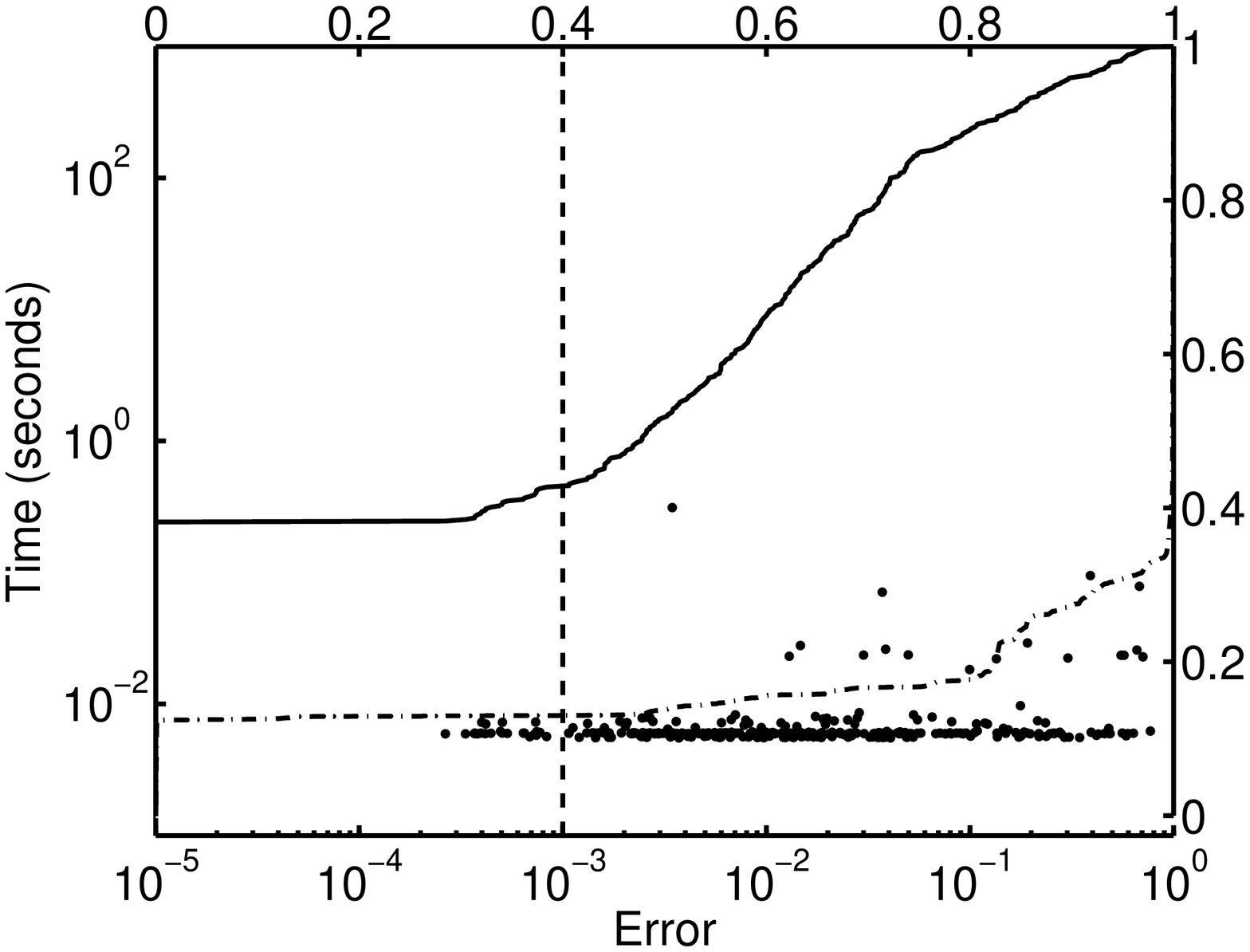} \\ {\tt chebfun} \end{minipage}
\begin{minipage}{5.7cm} \centering \includegraphics[width=5.7cm]{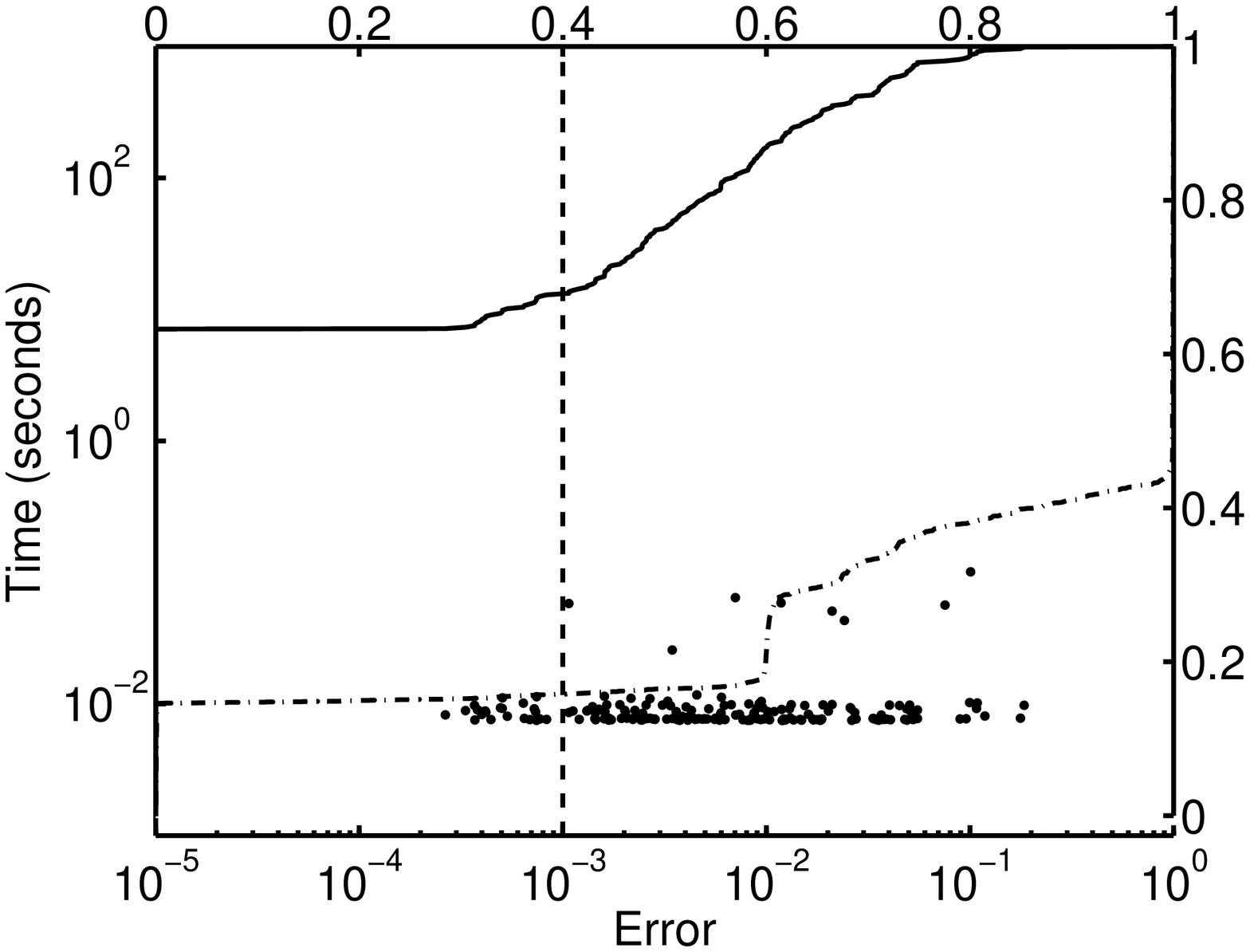} \\ {\tt chebfun}  (heavy duty) \end{minipage}
\begin{minipage}{5.7cm} \centering \includegraphics[width=5.7cm]{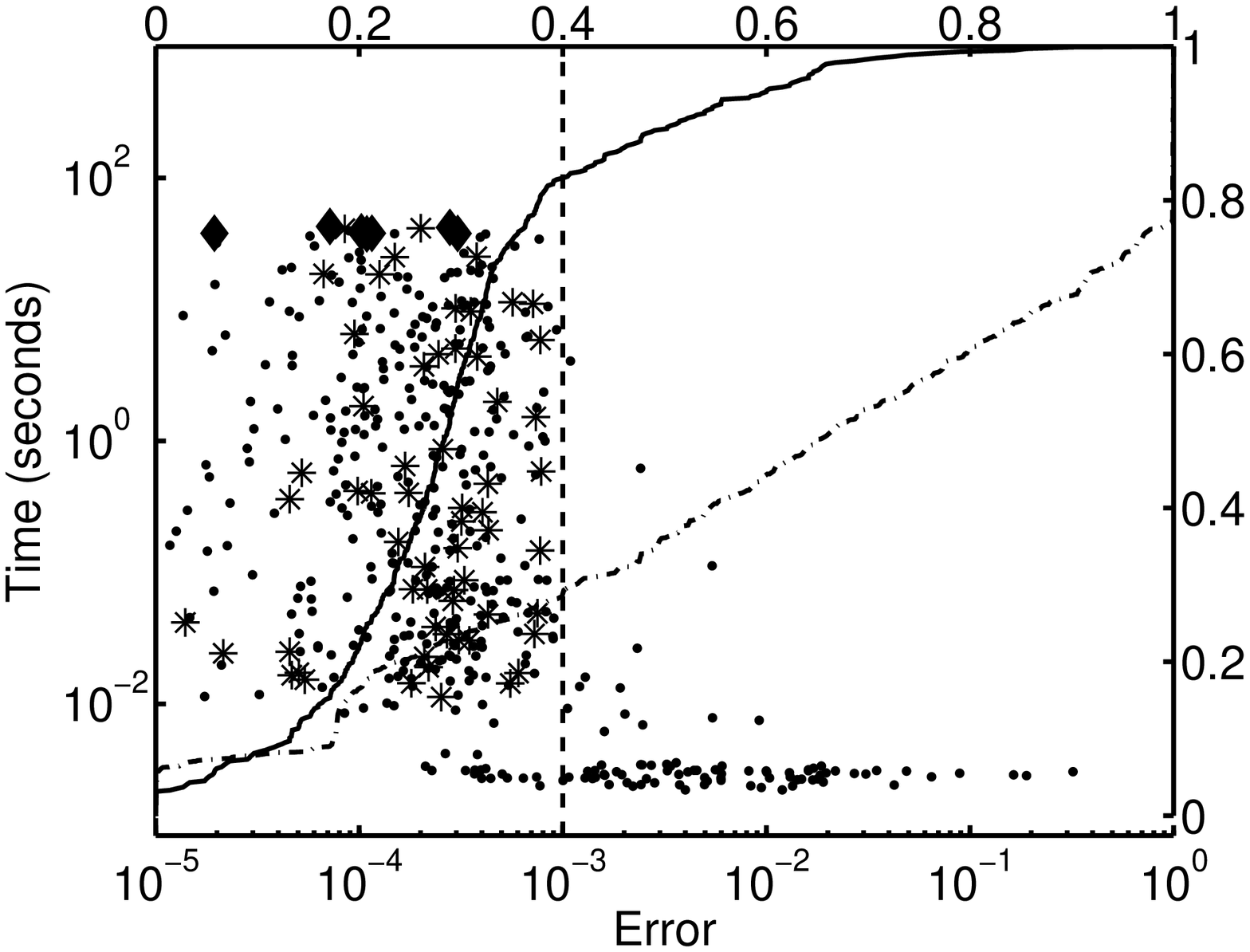} \\ Algorithm \ref{twostagealgo}  \end{minipage}
\begin{minipage}{5.7cm} \centering \includegraphics[width=5.7cm]{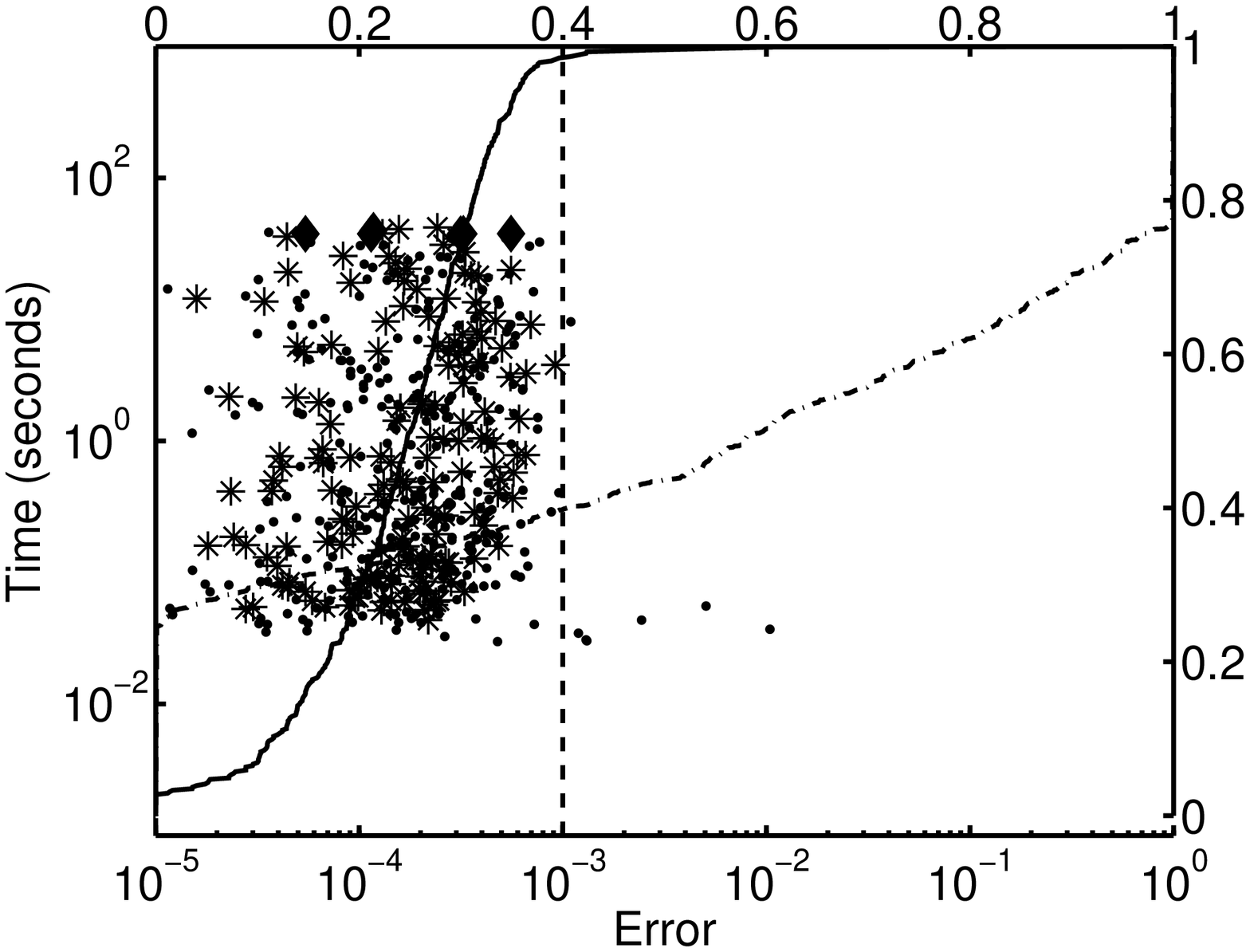} \\ Algorithm \ref{twostagealgo} (heavy duty)\end{minipage}
\begin{minipage}{5.7cm} \centering \includegraphics[width=5.7cm]{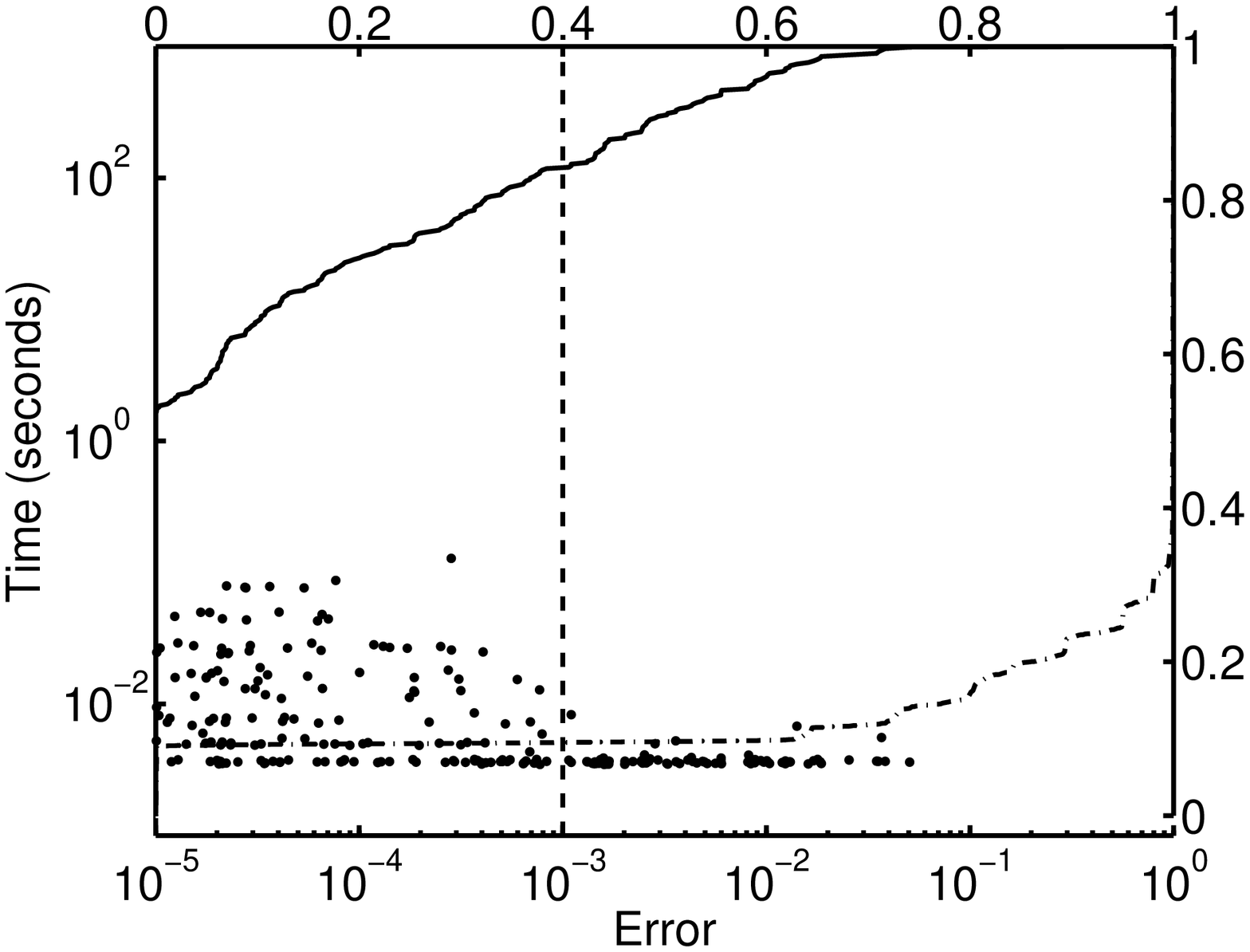} \\  Sobol'\end{minipage}
\begin{minipage}{5.7cm} \centering \includegraphics[width=5.7cm]{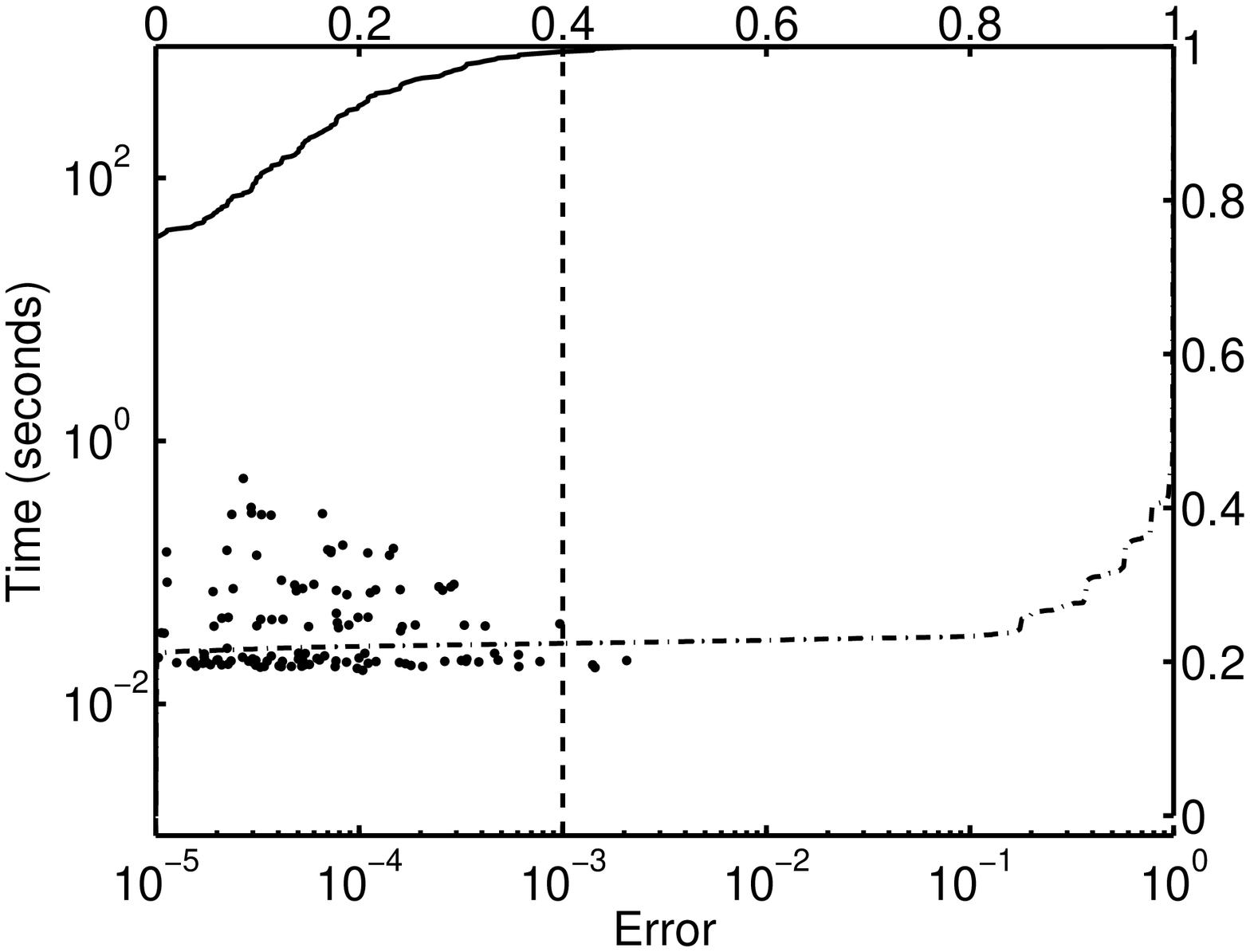} \\ Sobol' (heavy duty) \end{minipage}
\caption{Execution times and errors for test function \eqref{GaussianTestFun} for $d=1$ and error tolerance $\varepsilon=10^{-3}$, and a variety of parameters giving a range of $\sigma$ and $\tkappa$. Those points to the left/right of the dashed 
vertical line 
represent successes/failures of the automatic algorithms.  The solid line shows that cumulative distribution of actual errors, and the dot-dashed line shows the cumulative distribution of execution times.  For the Algorithm \ref{twostagealgo} the points labeled * are those for which the Corollary \ref{integcor} guarantees the error tolerance.\label{GaussianTestFunFig} }
\end{figure}

These $500$ random constructions of $f$ with $d=1$ are integrated using {\tt quad},  {\tt quadgk}, {\tt chebfun}, Algorithm \ref{twostagealgo}, and an automatic quasi-Monte Carlo algorithm that uses scrambled Sobol' sampling \citep{Owe95,Owe96,Owe97,Mat98,HonHic00a,DicPil10a}.  For the Sobol' sampling algorithm the error is estimated by an inflation factor of $1.1$ times the sample standard deviation of $8$ internal replicates of one scrambled Sobol' sequence \citep{Owe06a}.  The sample size is increased until this error estimate decreases to no more than the tolerance.  We 
have not yet found simple conditions on integrands for which this procedure is guaranteed to produce an estimate satisfying the error tolerance, 
and so we do not discuss it in detail.  We are however, intrigued by the fact that it does seem to perform rather well in practice.

For all but {\tt chebfun}, the specified absolute error tolerance is $\epsilon=0.001$.  The algorithm {\tt chebfun} attempts to do all calculations to near machine precision.  The observed error and execution times are plotted in Figure \ref{GaussianTestFunFig}.  Whereas {\tt chebfun} uses a minimum of $2^3+1=9$ function values, the figure labeled ``{\tt chebfun} (heavy duty)'' displays the results of requiring {\tt chebfun} to use at least $2^{8}+1=257$ function values.  Algorithm \ref{twostagealgo} takes $\alpha=0.01$, and $\fudge=1.1$.  For the plot on the left, $n_\sigma=2^{13}=8192$, which corresponds to  $\tkappa_{\max}=2.24$.  For the heavy duty plot on the right, $n_\sigma=2^{18}=262144$, which corresponds to  $\tkappa_{\max}=40.1$.  The same initial sample sizes are used for the Sobol' sampling algorithm.

Figure \ref{GaussianTestFunFig} shows that {\tt quad} and {\tt quadgk} are quite fast, nearly always providing an answer in less than $0.01$ seconds.  Unfortunately, they 
successfully meet the error tolerance only about $30\%$ of the time for {\tt quad} and $50$--$60\%$ of the time for {\tt quadgk}.  The difficult cases are those where $c_1$ is quite small, and these algorithms miss the sharp peak.  The performance of {\tt chebfun} is similar to that of {\tt quad} and {\tt quadgk}.  The heavy duty version of  {\tt chebfun} fares somewhat better.  For both of the {\tt chebfun} plots there are a significant proportion of the data that do not appear because their 
errors are smaller than $10^{-5}$.

In the plots for Algorithm \ref{twostagealgo} the alternative and somewhat costlier formula for $N_{\mu}$ in \eqref{NCBpracticaldef} is employed. An asterisk is used to label those points satisfying $\tkappa \le \tkappa_{\max}$, where $\tkappa$ is defined in \eqref{Lpnormdef}. All such points fall within the prescribed error tolerance,
which is even better than the guaranteed confidence of $99\%$.  
For Algorithm \ref{twostagealgo} (heavy duty) $\tkappa_{\max}$ is larger, so there are more points for which the guarantee holds.  Those points labeled with a dot, are those for which $\tkappa > \tkappa_{\max}$, and so no guarantee holds. The points labeled with a diamond are those for which Algorithm \ref{twostagealgo} attempts to exceed the cost budget that we set, i.e., it wants to choose $n_\mu$ such that $n_{\sigma}+n_\mu > N_{\max}:=10^9$. In these cases $n_\mu$ is chosen as $\lfloor 10^9 - n_\sigma \rfloor$, which often is still large enough to get an answer that satisfies the error tolerance. 
Algorithm \ref{twostagealgo} performs somewhat more robustly than {\tt quad}, {\tt quadgk}, and {\tt chebfun}, because it requires only a low degree of smoothness and takes a fairly large minimum sample. Algorithm \ref{twostagealgo} is generally much slower than the other algorithms because it does not assume any smoothness of the integrand. The more important point is that Algorithm \ref{twostagealgo} has a guarantee, whereas to our knowledge, the other routines do not.

From Figure \ref{GaussianTestFunFig}, the Sobol' sampling algorithm is more reliable and takes less time than Algorithm \ref{twostagealgo}.  This is due primarily to the fact that in dimension one, Sobol' sampling is equivalent to stratified sampling, where the points are more evenly spread than IID sampling.

Figure \ref{GaussianTestFunHDFig} repeats the simulation shown in Figure \ref{GaussianTestFunFig} for the same test function \eqref{GaussianTestFun}, but now with $d=2, \ldots, 8$ chosen randomly and uniformly.  For this case the univariate integration algorithms are inapplicable, but the multidimensional routines can be used.  There are more cases where the Algorithm \ref{twostagealgo} tries to exceed the maximum sample size allowed, i.e., $(n_{\sigma}+n_\mu)d > N_{\max}:=10^9$, but the behavior seen for $d=1$ still generally applies.  

\begin{figure}
\centering
\begin{minipage}{5.7cm} \centering \includegraphics[width=5.7cm]{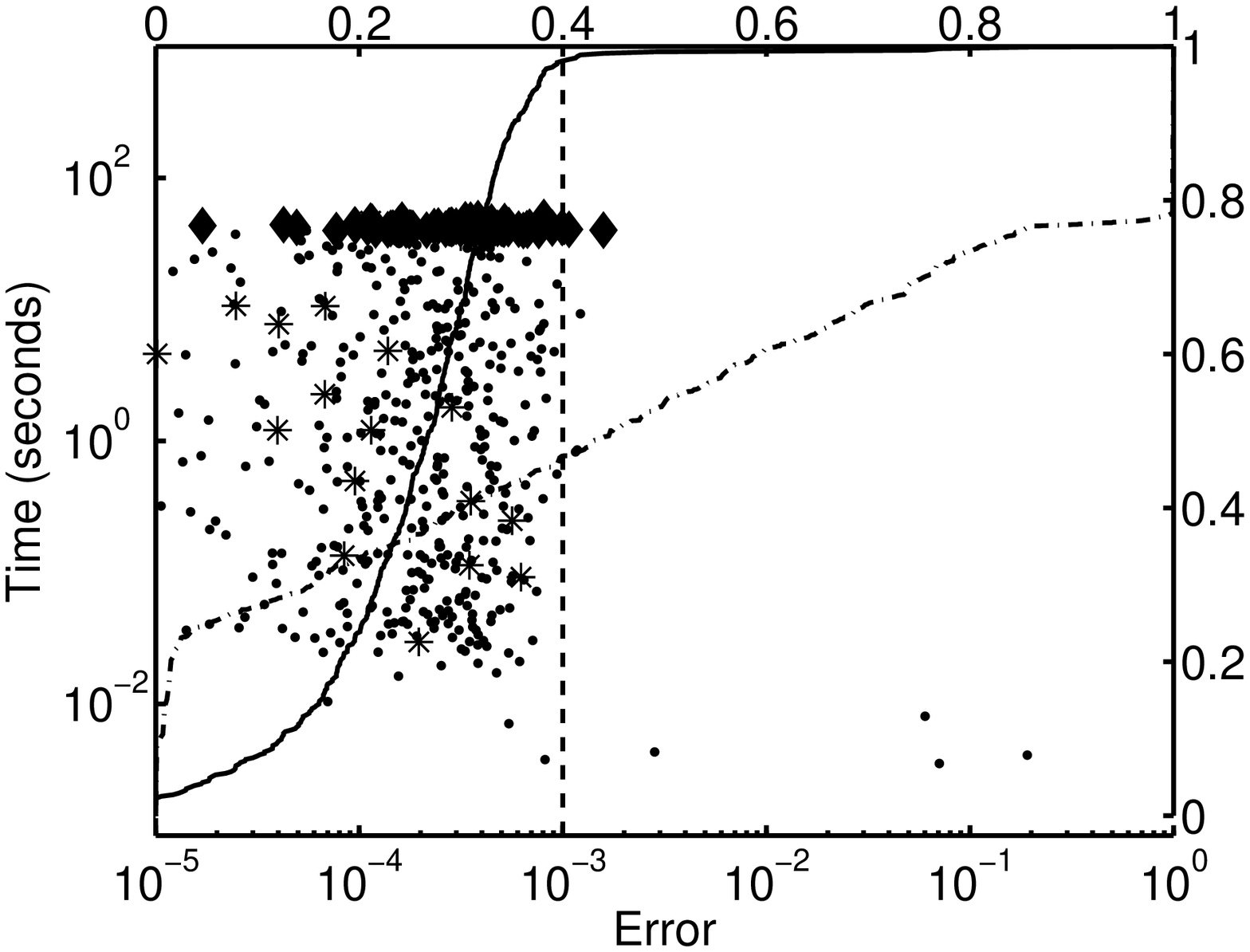} \\ Algorithm \ref{twostagealgo} \end{minipage}
\begin{minipage}{5.7cm} \centering \includegraphics[width=5.7cm]{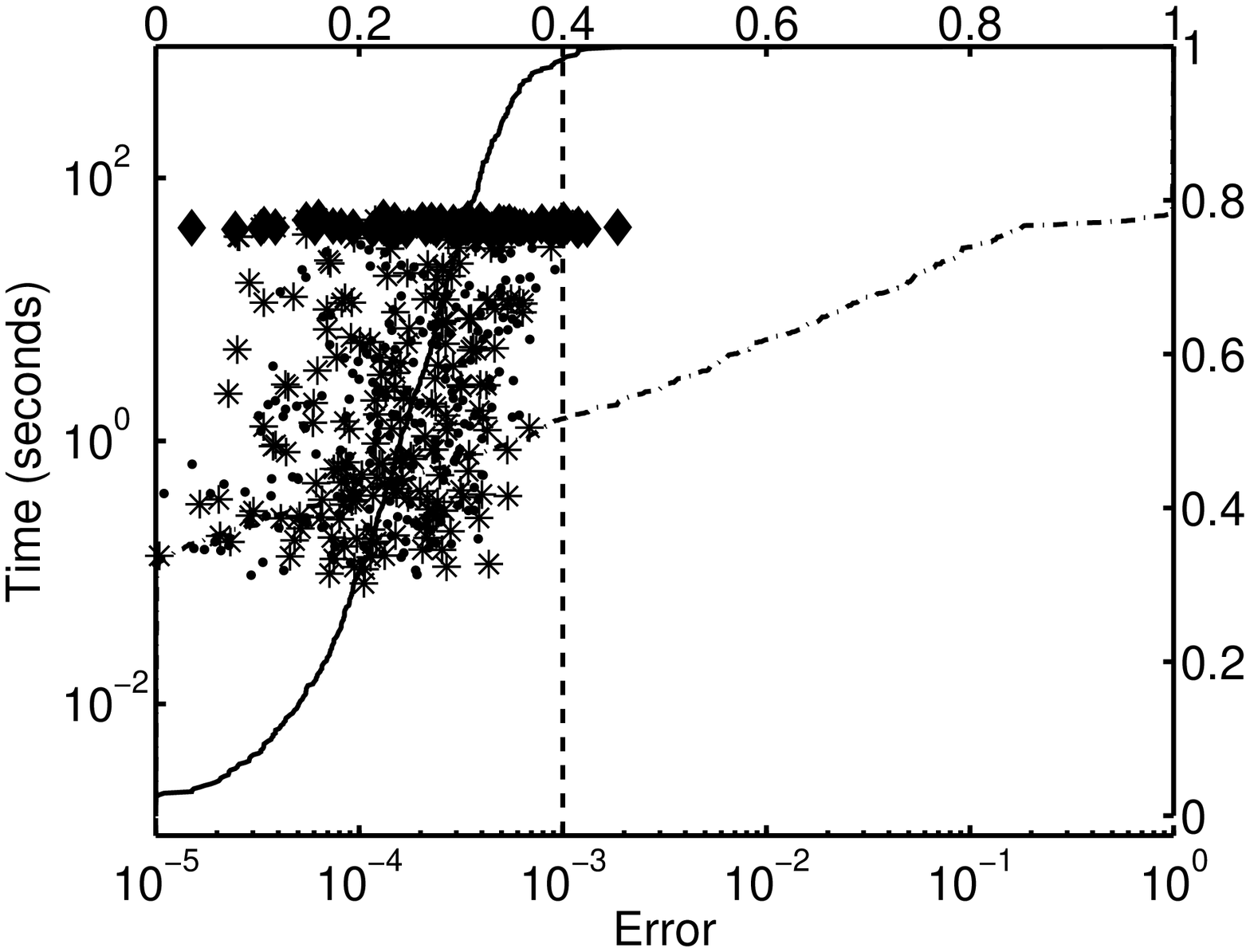} \\ Algorithm \ref{twostagealgo} (heavy duty) \end{minipage}
\begin{minipage}{5.7cm} \centering \includegraphics[width=5.7cm]{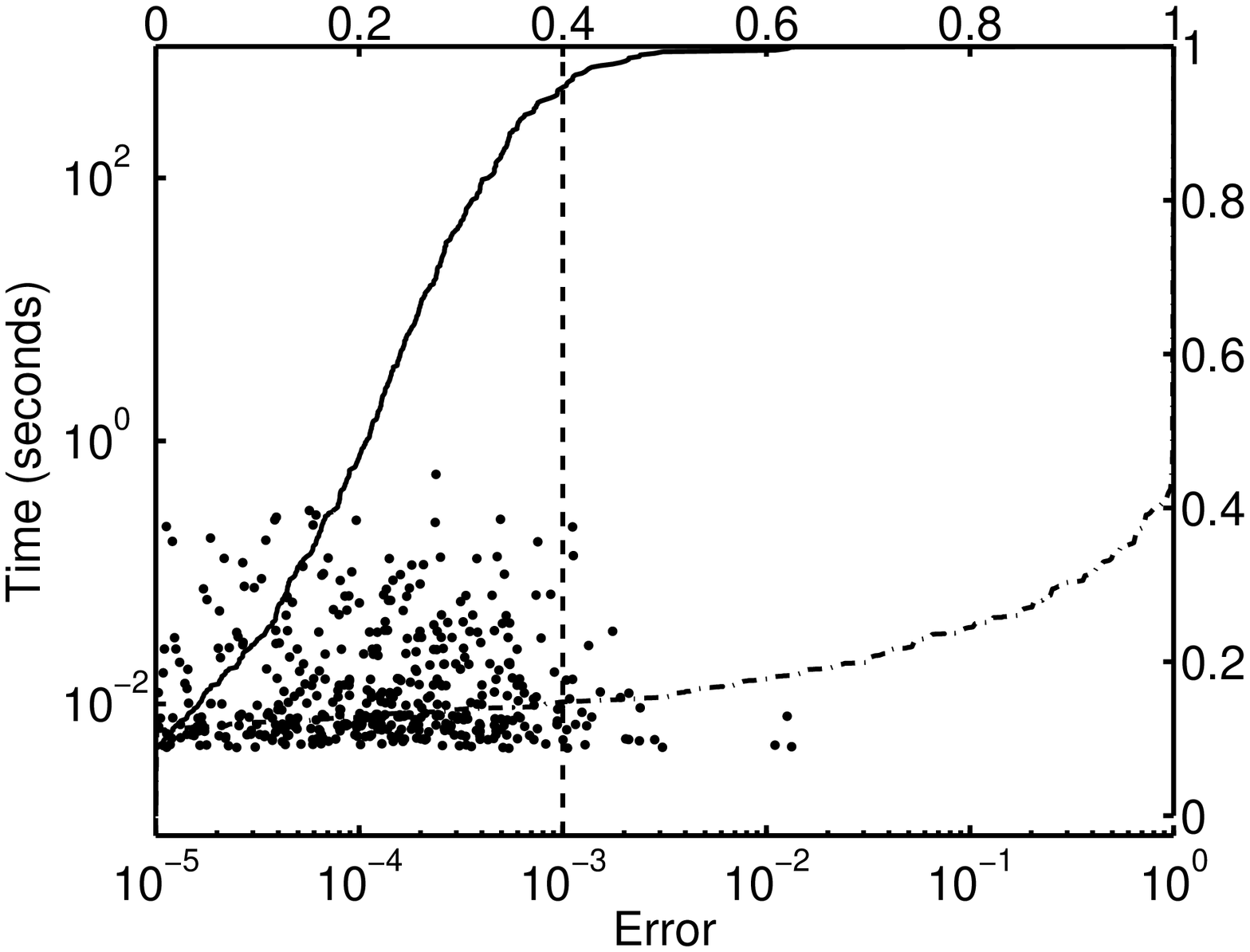} \\ Sobol' \end{minipage}
\begin{minipage}{5.7cm} \centering \includegraphics[width=5.7cm]{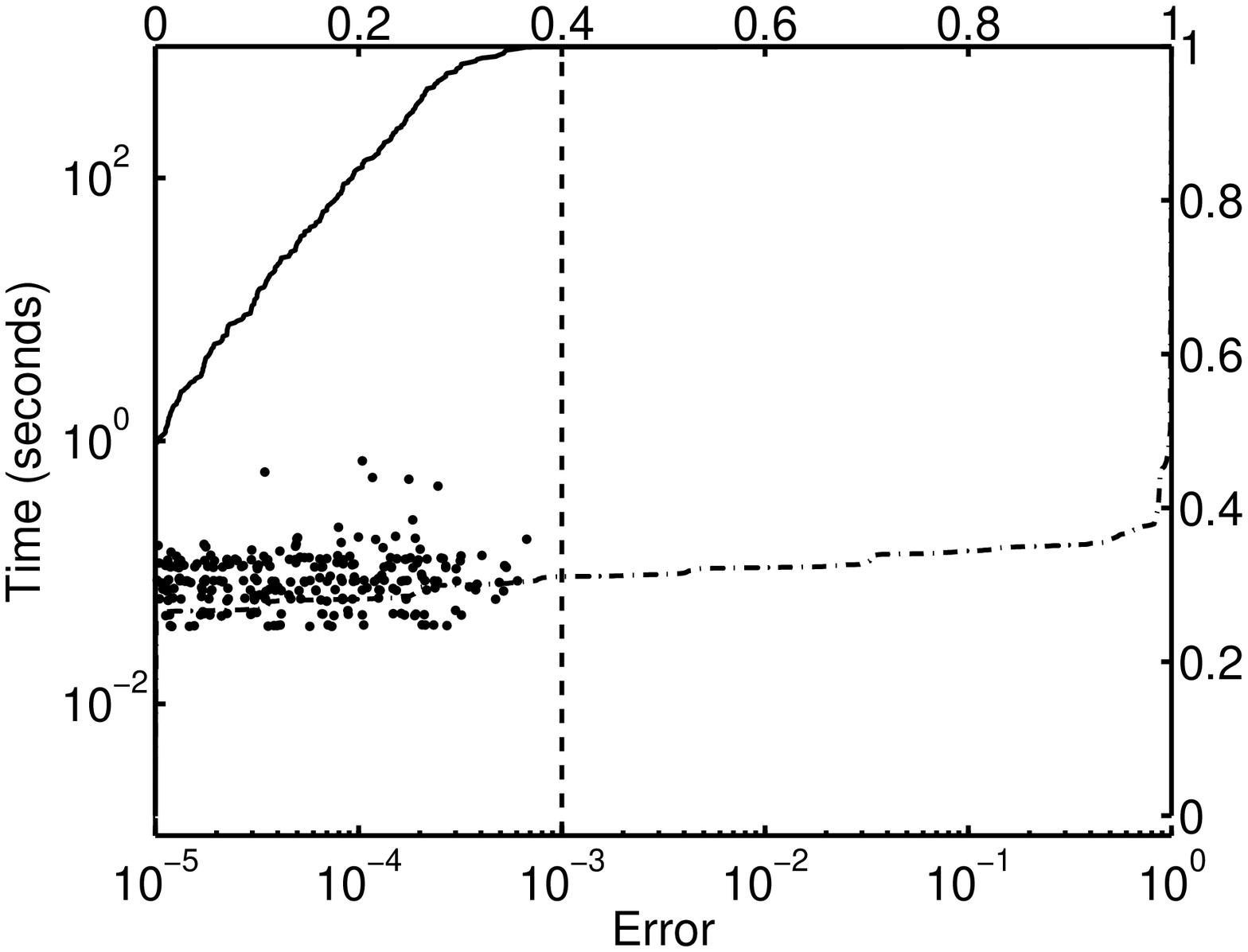} \\ Sobol' (heavy duty) \end{minipage}
\caption{Execution times and errors for test function \eqref{GaussianTestFun} for $d=2, \ldots, 8$ and $\varepsilon=10^{-3}$, with the rest of the parameters as in Figure \ref{GaussianTestFunFig}.\label{GaussianTestFunHDFig}}
\end{figure}

\subsection{Asian Geometric Mean Call Option Pricing}

The next example involves pricing an Asian geometric mean call option.  Suppose that the price of a stock $S$ at time $t$ follows a geometric Brownian motion with constant interest rate, $r$, and constant volatility, $v$.  
One may express the stock price in terms of the initial condition, $S(0)$, as 
\[
S(t)=S(0) \exp[(r-v^2/2)t + v B(t)], \qquad t \ge 0,
\]
where $B$ is a standard Brownian motion.  
The discounted payoff of the Asian geometric mean call option with an expiry of $T$ years, a strike price of $K$, and assuming a discretization at $d$ times is 
\begin{equation} \label{payoff}
Y=\max\biggl([\sqrt{S(0)}S(T/d) S(2T/d)\cdots S(T(d-1)/d) \sqrt{S(T)}]^{1/d} - K,0 \biggr)\me^{-rT}.
\end{equation}
The fair price of this option is $\mu=\e(Y)$. One of our chief reasons for choosing this option for numerical experiments is that its price can be computed analytically, while the numerical computation is non-trivial.

In our numerical experiments, the values of the Brownian motion at different times required for evaluating the stock price, $B(T/d), B(2T/d), \ldots,  B(T)$, are computed via a  Brownian bridge construction.  This means that for one instance of the Brownian motion we first compute $B(T)$, then $B(T/2)$, etc., using independent Gaussian random variables $X_1, \ldots, X_d$, suitably scaled. The Brownian bridge accounts for more of the low frequency motion of the stock price by the $X_j$ with smaller $j$, which allows the Sobol' sampling algorithm to do a better job.  

The option price, $\mu=\e(Y)$, is approximated by Algorithm \ref{twostagealgo} and the Sobol' sampling algorithm using an error tolerance of $\varepsilon=0.05$, and compared to the analytic value of $\mu$.  The result of $500$ replications is given in Figure \ref{AsianGeoMeanCallFig}.  Some of the parameters are set to be fixed values, namely,
\[
S(0)=K=100, \qquad T=1, \qquad r=0.03.
\]
The volatility, $v$, is drawn uniformly between $0.1$ and $0.7$.  The number of time steps, $d$, is chosen to be uniform over $\{1, 2, 4, 8, 16, 32\}$.  The true value of $\mu$ for these parameters is between about $2.8$ and $14$.

\begin{figure}
\centering
\begin{minipage}{5.7cm} \centering \includegraphics[width=5.7cm]{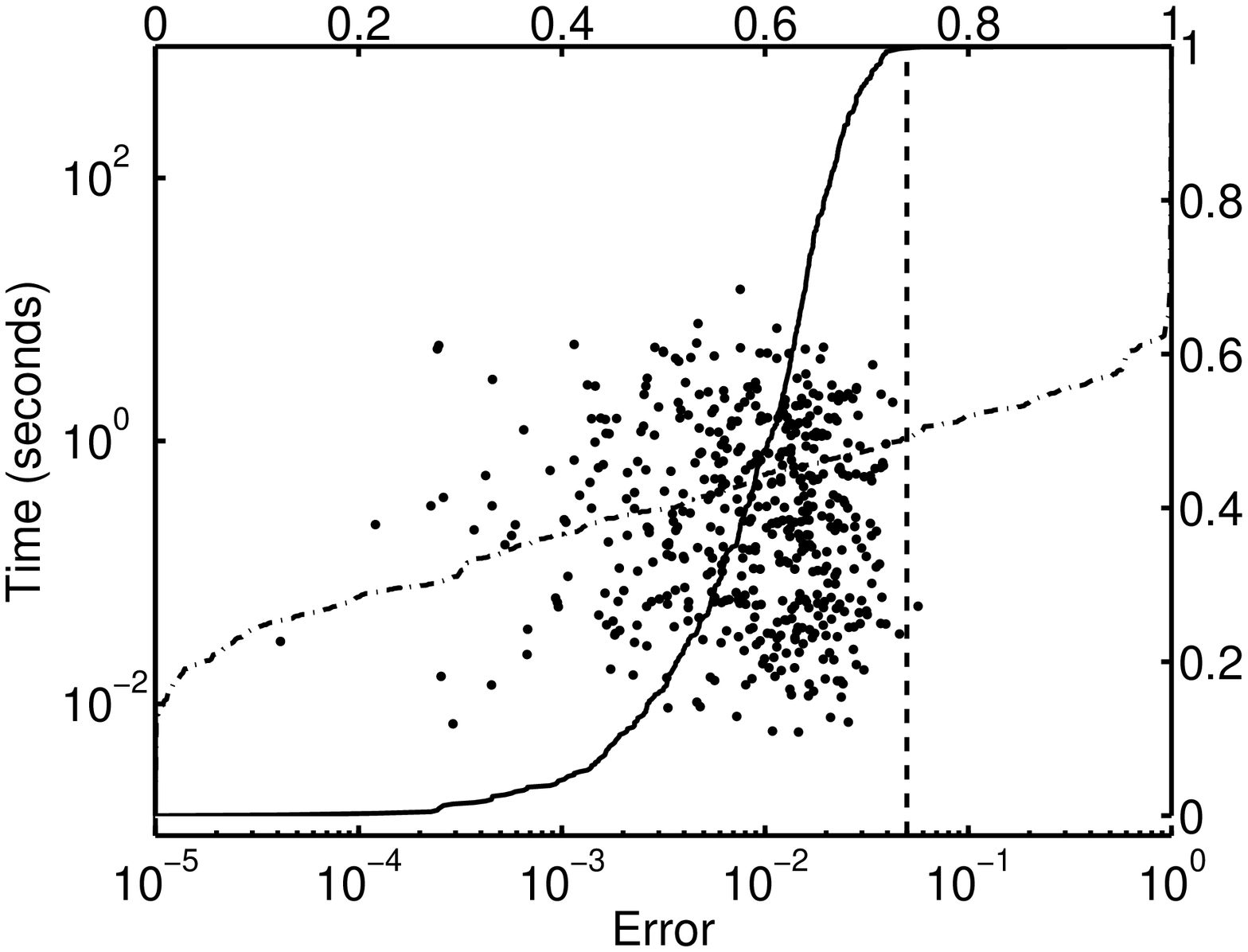} \\ Algorithm \ref{twostagealgo} \end{minipage}
\begin{minipage}{5.7cm} \centering \includegraphics[width=5.7cm]{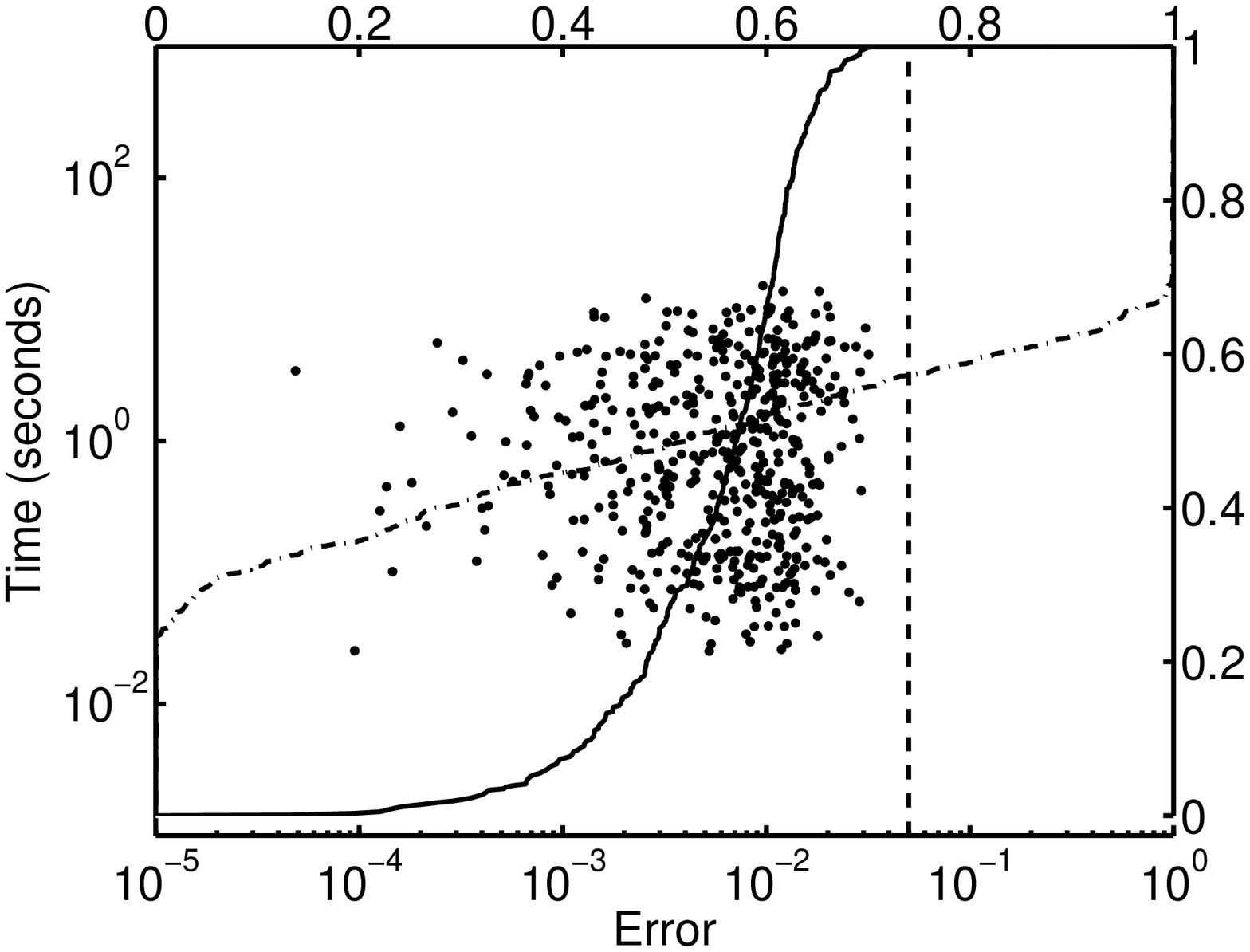} \\ Algorithm \ref{twostagealgo} (heavy duty)\end{minipage}
\begin{minipage}{5.7cm} \centering \includegraphics[width=5.7cm]{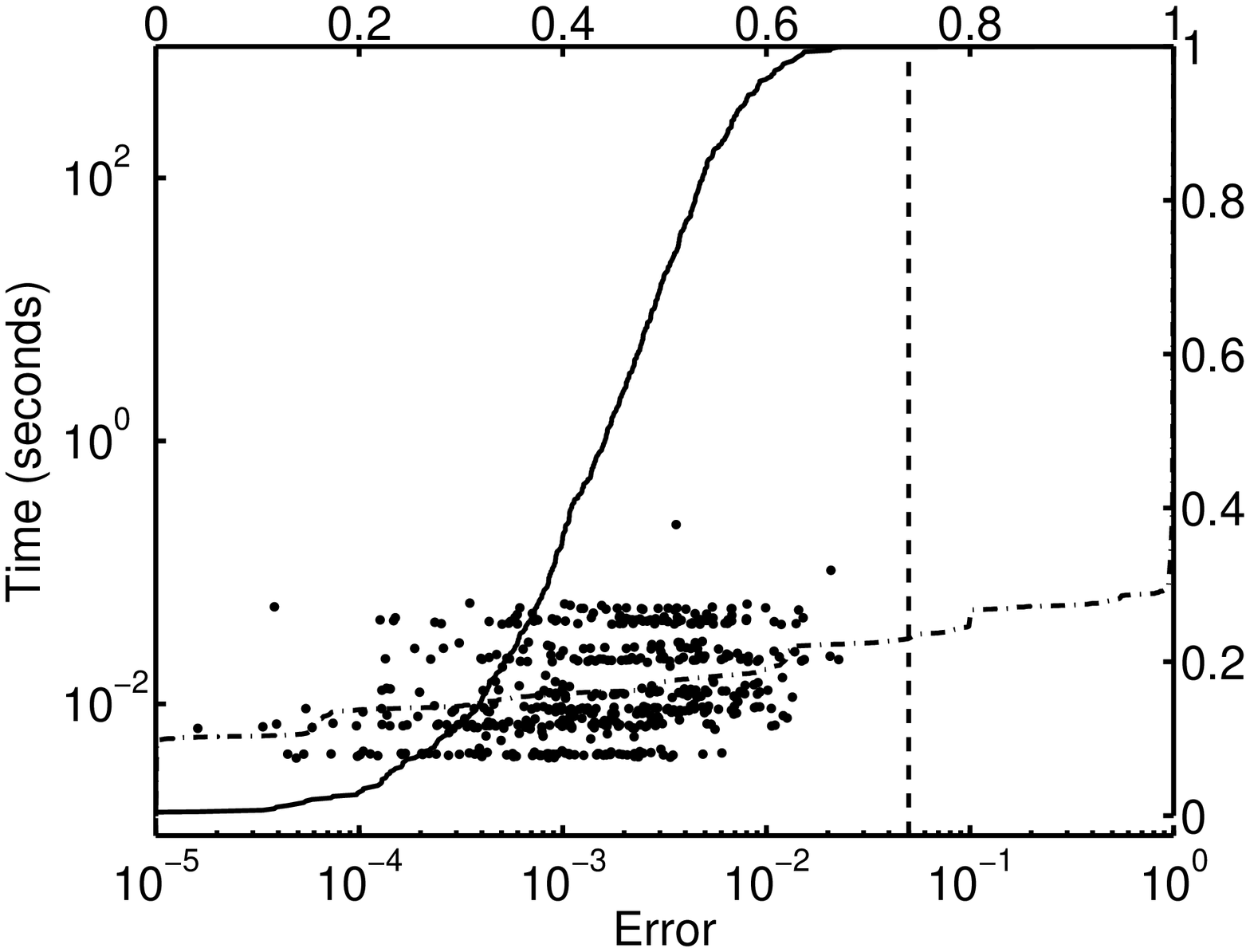} \\ Sobol' \end{minipage}
\begin{minipage}{5.7cm} \centering \includegraphics[width=5.7cm]{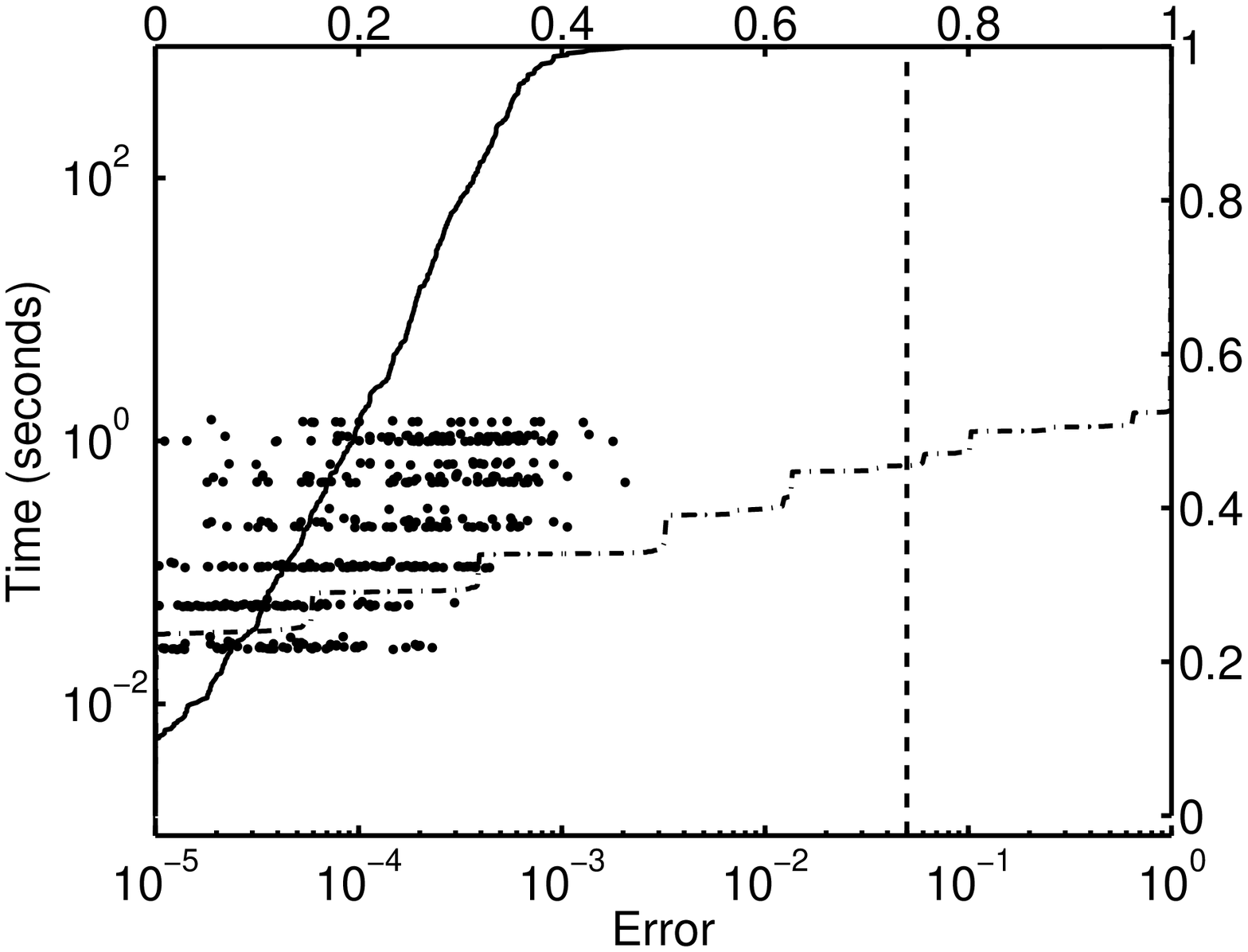} \\ Sobol' (heavy duty) \end{minipage}
\caption{Execution times and errors for the Asian geometric mean call option for $d=1, 2, 4, 8, 16, 32$ and $\varepsilon=0.05$.\label{AsianGeoMeanCallFig}}
\end{figure}

For this example the true kurtosis of $Y$ is unknown.  
Both  Algorithm \ref{twostagealgo} and the  Sobol' sampling algorithm compute the option price to the desired error tolerance with high reliability.  For the IID sampling Algorithm \ref{twostagealgo} and the ordinary Sobol' sampling algorithm it can be seen that some of the errors are barely under the error tolerance, meaning that the sample size is not 
chosen too conservatively.  For the heavy duty Sobol' algorithm, the high initial sample size seems to lead to smaller than expected errors and larger than necessary 
computation times.

\section{Discussion} \label{discusssec}

Practitioners often construct CLT-based confidence intervals with the true variance estimated by the sample variance, perhaps multiplied by some inflation factor.  Often, this approach works, but it has no guarantee of success.  The two-stage algorithm presented here is similar to the approach just described, but it carries guarantees. These are derived by employing Cantelli's inequality to ensure a reliable variance upper bound, and by employing a Berry-Esseen inequality to ensure a large enough sample for the sample mean.  

In certain cases our procedure multiplies the computational cost by a large factor such as $2$ or $10$ or even $100$ compared to what one might spend based on the CLT with a known value of $\sigma$ (see Figure \ref{Costfig}).  While this seems inefficient, one should remember that the total elapsed time may still be well below several seconds.  Furthermore, one typically does not know $\sigma$ in advance, and our adaptive algorithm estimates $\sigma$ and then an appropriate sample size $n_\mu$ from the data.  Our algorithmic cost will be low when the unknown $\sigma$ is small and large when $\sigma$ is large.

Like any algorithm with guarantees, our algorithm does need to make assumptions about the random variable $Y$.  We assume a known bound on the kurtosis of $Y$, either specified directly or implied by the user's choice of the sample size for estimating the variance, $n_\sigma$, and the variance inflation factor, $\fudge^2$.  This is a philosophical choice.  We prefer not to construct an algorithm that assumes a bound on the variance of 
$Y$, because such an algorithm would not be guaranteed for $cY$ with $\abs{c}$ large enough.  If our algorithm works for $Y$, it will also work for $cY$, no matter how large $\abs{c}$ is. 

In practice the user may not know a priori if $\tkappa \le \tkappa_{\max}$ since it is even more difficult to estimate $\tkappa$ from a sample than it is to estimate $\sigma^2$.  Thus, the choice of $\tkappa_{\max}$ relies on the user's best judgement.  Here are a few thoughts that might help.  One might try a sample of typical problems for which one knows the answers and use these problems to suggest an appropriate $\tkappa_{\max}$.  Alternatively, one may think of $\tkappa_{\max}$ not as a parameter to be prescribed, but as a reflection of the robustness of one's Monte Carlo algorithm having chosen $\alpha$, $n_\sigma$ and $\fudge$.  The discussion at the end of Section \ref{algcostsubsec} provides guidance on how to choose $n_\sigma$ and $\fudge$ to achieve a given $\tkappa_{\max}$ in a manner that minimizes total computational cost.  Briefly, one should not skimp on $n_\sigma$, but choose $n_\sigma$ to be several thousand times $\tkappa_{\max}$ and employ a $\fudge$ that is relatively close to unity.  Another way to look at the Theorem \ref{mainadaptthm} is that, like a pathologist, it tells you what went wrong if the two-stage adaptive algorithm fails: the kurtosis of the random variable must have been too large.
In any case, as one can see in Figure \ref{Costfig}, in the limit of vanishing $\varepsilon/\sigma$, i.e., $N_{\text{CLT}} \to \infty$, the choice of $\tkappa_{\max}$ makes a negligible contribution to the total cost of the algorithm.  The main determinant of computational cost is $\varepsilon/\sigma$.

\cite{BahSav56} prove in Corollary 2 that it is \emph{impossible} to construct exact confidence intervals for the mean of random variable whose distribution lies in a set satisfying a few assumptions. One of these assumptions is that the set of distributions is convex.  This assumption is violated by our assumption of bounded kurtosis in Theorem \ref{mainadaptthm}. Thus, we are able to construct guaranteed confidence intervals.

Our algorithm is adaptive because $n_\mu$ is determined from the sample variance.  Information-based complexity theory tells us that adaptive information does not help for the integration problem for symmetric, convex sets of integrands, $f$, in the worst case and probabilistic settings \citep[Chapter 4, Theorem 5.2.1; Chapter 8, Corollary 5.3.1]{TraWasWoz88}.  Here, in Corollary \ref{integcor} the cone, $\cc_{\tkappa_{\max}}$, although symmetric, is not a convex set, so it is possible for adaption to help.

There are a couple of areas that suggest themselves for further investigation.  One is relative error, i.e., a fixed width confidence interval of the form 
\[
\Pr[\abs{\mu-\hmu} \le \varepsilon \abs{\mu}] \ge 1-\alpha.
\]
Here the challenge is that the right hand side of the first inequality includes the unknown mean.

Another area for further work is to provide guarantees for automatic quasi-Monte Carlo algorithms. Here the challenge is finding reliable formulas for error estimation.  Typical error bounds involve a semi-norm of the integrand that is harder to compute than the original integral.  For randomized quasi-Monte Carlo an estimate of the variance of the sample mean using $n$ samples does not tell you much about the variance of the sample mean using a different number of samples.

\section*{Acknowledgements} The authors gratefully acknowledge discussions with Erich Novak and Henryk Wo\'zniakowski, and the comments of the referees.  
The plots of the univariate fooling functions were prepared with the help of Nicholas Clancy and Caleb Hamilton.  The first and fourth authors would like to express their thanks to the local organizers of the Tenth International Conference on Monte Carlo and Quasi-Monte Carlo Methods in Scientific Computing for hosting a wonderful conference.

\bibliographystyle{spbasic}
\bibliography{FJH22,FJHown22}
\end{document}